\documentclass[11pt]{article}
\usepackage{amsfonts}
\usepackage{xcolor}
\usepackage{amsmath,amssymb,comment}
\newtheorem{theo}{Theorem}[section]
\newtheorem{lem}[theo]{Lemma}

\newtheorem{defi}[theo]{Definition}

\newcommand{\mysection}[1]{\section{#1} \setcounter{equation}{0}}
\newcommand{\proof}{{\sc Proof.} \quad}
\newcommand{\proofc}{{\sc Proof} \ }
\newcommand{\be}{\begin{equation} \label}
\newcommand{\ee}{\end{equation}}
\newcommand{\bea}{\begin{eqnarray}\label}
\newcommand{\eea}{\end{eqnarray}}
\newcommand{\bas}{\begin{eqnarray*}}
\newcommand{\eas}{\end{eqnarray*}}
\newcommand{\bit}{\begin{itemize}}
\newcommand{\eit}{\end{itemize}}
\newcommand{\qed}{\hfill$\Box$ \vskip.2cm}
\newcommand{\nn}{\nonumber}

\newcommand{\R}{\mathbb{R}}
\newcommand{\N}{\mathbb{N}}
\newcommand{\pO}{\partial\Omega}
\newcommand{\eps}{\varepsilon}

\newcommand{\wsto}{\stackrel{\star}{\rightharpoonup}}

\newcommand{\io}{\int_\Omega}
\newcommand{\na}{\nabla}
\newcommand{\Del}{\Delta}
\newcommand{\del}{\delta}

\newcommand{\lam}{\lambda}
\newcommand{\Lam}{\Lambda}

\newcommand{\pa}{\partial}
\newcommand{\bom}{\overline{\Omega}}
\newcommand{\Om}{\Omega}
\newcommand{\om}{\omega}

\newcommand{\ov}{\overline}
\newcommand{\un}{\underline}

\newcommand{\hs}{\hspace*}
\newcommand{\sm}{\setminus}

\newcommand{\vp}{\varphi}
\newcommand{\lbal}{\left\{ \begin{array}{l}}
\newcommand{\lball}{\left\{ \begin{array}{ll}}
\newcommand{\ear}{\end{array} \right.}

\newcommand{\abs}{\\[5pt]}
\newcommand{\Abs}{\\[5mm]}

\newcommand{\ou}{\ov{u}}

\newcommand{\ovv}{\ov{v}}

\newcommand{\ueps}{u_\eps}
\newcommand{\veps}{v_\eps}

\newcommand{\zd}{\zeta_\del}
\newcommand{\LL}{L}
\begin{document}
\allowdisplaybreaks
%
%
%
\title{Uniqueness and large-time continuous dependence\\
in a doubly degenerate nutrient taxis system}
\author{
Michael Winkler\footnote{michael.winkler@math.uni-paderborn.de}\\
{\small Institut f\"ur Mathematik, Universit\"at Paderborn,}\\
{\small 33098 Paderborn, Germany} }
\date{}
\maketitle
\begin{abstract}
\noindent 
The chemotaxis system
\bas
	\lbal
	u_t = \nabla \cdot (uv\nabla u) - \chi \nabla\cdot (u^2 v\nabla v) + \ell uv, \\[1mm]
	v_t = \Delta v - uv,
	\ear
\eas
is considered in a smoothly bounded domain $\Omega\subset\R^n$ and with $\chi\ge 0$ and $\ell\ge 0$.\abs
Within a suitable concept of solvability, a statement on uniqueness of solutions emanating from initial data $(u_0,v_0)$
is derived under the mere assumption that $u_0\in L^\infty(\Omega)$ and $v_0\in W^{1,\infty}(\Omega)$ are such that
$u_0\ge 0$ and $v_0>0$ in $\overline{\Omega}$; in particular, no condition on continuity of $u_0$ is imposed.\abs
This result is thereafter used to reveal that in physcally meaningful scenarios,
for any initial data which are additionally assumed to be continuous and nontrivial,
such solutions can always be found, and that these solutions depend continuously on the initial distributions 
not only in bounded sets of the form $\Om\times (0,T)$, but actually 
throughout $\Omega\times (0,\infty)$, in a suitable sense.\abs
{\bf Key words:} chemotaxis, uniqueness, continuous dependence, degenerate diffusion\\
{\bf MSC 2020:} 35A02 (primary); 35B30, 35K65, 92C17 (secondary)
\end{abstract}
\newpage
\section{Introduction}\label{intro}
This manuscript is concerned with the problem
\be{0}
	\lball
	u_t = \nabla \cdot (uv\nabla u) - \chi \nabla\cdot (u^2 v\nabla v) + \ell uv, 
	\qquad & x\in\Omega, \ t>0, \\[1mm]
	v_t = \Delta v - uv,
	\qquad & x\in\Omega, \ t>0, \\[1mm]
	(uv\nabla u - \chi u^2 v\nabla v)\cdot \nu = \nabla v\cdot \nu=0,
	\qquad & x\in\pO, \ t>0, \\[1mm]
	u(x,0)=u_0(x), \quad
	v(x,0)=v_0(x),
	\qquad & x\in\Omega,
	\ear
\ee
to be considered with parameters $\chi\ge 0$ and $\ell\ge 0$ and with nonnegative initial data $u_0$ and $v_0$,
in a smoothly bounded domain $\Om\subset\R^n$, $n\ge 1$.
The PDE system in (\ref{0}) has been proposed as a model for the spatio-temporal evolution in populations of 
{\em Bacillus subtilis} in nutrient-poor constellations
(\cite{plaza}, \cite{plaza_JMB}, \cite{kawasaki_JTB1997}); in fact, it is known from experiments that in such scenarios,
strikingly complex patterns can form spontaneously, turning initially unstructured distributions into 
states of fractal-like complexity (\cite{matsushita_fujikawa_PHYSA1990}, \cite{fujikawa_PHYSA1992}).\abs
From a perspective of mathematical analysis, evident challenges are due to the circumstance
that (\ref{0}) assumes random diffusion in the population density $u=u(x,t)$ to be weakened not only near low levels
of $u$ itself, but also in regions where the nutrient concentration $v=v(x,t)$ is small. 
On the one hand, an accordingly implied lack of knowledge on solution regularity brings about some considerable 
obstacles for the development of solution theories for (\ref{0}), which indeed seem to date much less complete than
for related non-degenerate variants such as the classical Keller-Segel-consumption system
\be{KSc}
	\lbal
	u_t = \Del u - \chi \na\cdot (u\na v), \\[1mm]
	v_t = \Del v - uv.
	\ear
\ee
Indeed, no-flux boundary value problems for (\ref{KSc}) are known to admit globally bounded smooth solutions for initial data
of arbitrary size when $n\le 2$ (\cite{songmu_zheng}, \cite{taowin_consumption}), and in the case when $n\ge 3$
it is possible to find at least some global weak solutions (\cite{lankeit_win_NA}) which eventually become smooth and classical
if $n=3$ (\cite{taowin_consumption});
in contrast to this, results available for the doubly degenerate problem (\ref{0}) seem limited to constructions
of some bounded weak solutions when $n\ge 2$ (\cite{yuxiang_li_M3AS}, \cite{zhang_li_preprint1}, \cite{win_wu}, \cite{liwin_CPAA}), 
while in three-dimensional frameworks such solutions so far have been found to exist only under appropriate smallness
assumptions on $v_0$ (\cite{win_wu2}).
Notably, however, considerable progress with regard to the issues of global existence and boundedness 
has recently been achieved for some closely related variants
in which certain saturation effects in cross-diffusion or signal absorption are presupposed
(\cite{duan_wu}, \cite{duan_wu_arxiv}, \cite{zhang_li_preprint2}, \cite{yifu_wang2025}, \cite{yifu_wang_AML}, 
\cite{genglin_JDE}, \cite{black_kohatsu_wu}),
or logistic-type zero-order damping is taken into account 
(\cite{duan_ch6}, \cite{xu_pan}, \cite{liwin_PROCA}, \cite{yuxiang_li_AMOP2025}, \cite{yuting});
even some extensions to more complex and partially less dissipative settings have been possible
(\cite{yifu_wang_SIMA2026}), \cite{hervas}, \cite{hervas2}).
We remark here that in recent literature, a noticeable number of farther relatives involving signal-dependent degeneracies in 
contexts of taxis-type cross-diffusion systems has been studied; a small selection of corresponding results can be found in
\cite{fujie_jiang_CVPDE}, \cite{fujie_senba_BU}, \cite{jin_kim_wang}, \cite{genglin_yuan}, \cite{laurencot_CPDE}, 
\cite{desvillettes}, \cite{liangchen}, \cite{ahn_yoon} and \cite{lee_yoon}, for instance.\abs
On the other hand, the inclusion of doubly degenerate diffusion substantially distinguishes (\ref{0}) from (\ref{KSc}) 
with respect to the potential to describe structure formation processes in the flavor of those reported in 
\cite{matsushita_fujikawa_PHYSA1990} and \cite{fujikawa_PHYSA1992}.
While the solution behavior in (\ref{KSc}) seems exclusively dominated by large time stabilization toward semitrivial
homogeneous equilibria, namely, several rigorous detections of nontrivial large time behavior and of stability properties
enjoyed by arbitrary steady states have indicated a considerable tendency of (\ref{0}) and some variants toward 
supporting complex structures 
(cf., e.g., \cite{duan_wu_arxiv}, \cite{duan_ch6}, \cite{yifu_wang2025}, \cite{win_wu}, \cite{win_wu2}, \cite{win_TRAN}).\abs
{\bf Main results I: Uniqueness of locally bounded weak solutions.} \quad
In the present manuscript, the first object of study is the question how far solutions to (\ref{0}) can be expected to be unique.
In fact, due to the diffusion degeneracies in (\ref{0}) this issue seems not accesible to approaches well-established
in the analysis neither of parabolic systems in general (\cite{amann}), nor of triangular cross-diffusion systems of taxis type
in particular (\cite{horstmann_win}). 
In line with this, the literature in this regard so far seems restricted to a statement on uniqueness in a one-dimensional
system version in which, after passing to primitives, corresponding arguments can operate at levels of integrated densities
(\cite{chen_unique}).\abs
To develop an approach capable of coping with higher-dimensional frameworks, we will base our strategy on 	
a duality-type reasoning reminiscent of classical findings on uniqueness in porous medium type problems (\cite{ACP}). 
Compared to cases of Keller-Segel systems involving diffusion operators of the form $\na\cdot (u^{m-1}\na u)$ with $m>1$
(\cite{ki_ahm_lee}, \cite{sugiyama_kawakami_JDE}), however, for a corresponding treatment of (\ref{0}) a new design seems
necessary due to the additional presence of cross-degeneracies referring to the quantity $v$.
Our analysis in this regard, to be developed in detail in Section \ref{sect2}, will be centered around solutions
$(\vp,\psi)$ to non-degenerate approximations of the backward problems which for given solutions $(u,v)$ and $(\ou,\ovv)$
of (\ref{0}) in $\Om\times (0,T)$, and for $T_0\in (0,T)$, are given by
\be{01}
	\lball
	\vp_t
	= - \frac{1}{2} (u+\ou) v \Del \vp
	- \frac{1}{2}(u+\ou)\na v\cdot\na\vp \\
	\hs{14mm}
	- \chi(u+\ou) v \na v\cdot\na\vp 
	- \ell v \vp
	+ v \psi,
	\qquad & x\in \Om, \ t\in (0,T_0), \\[1mm]
	\psi_t 
	= -\Del\psi
	+ \chi \ou^2 \ovv \Del\vp
	+ \frac{1}{2} \na \ou^2 \cdot\na\vp \\
	\hs{14mm}
	- \chi \ou^2 \na v\cdot\na\vp
	+ \chi \ovv \na \ou^2 \cdot\na\vp
	+ \chi \ou^2 \na\ovv\cdot\na\vp \\
	\hs{14mm}
	- \ell \ou \vp
	+ \ou \psi,
	\qquad & x\in \Om, \ t\in (0,T_0), \\[1mm]
	\frac{\pa\vp}{\pa\nu}=\frac{\pa\psi}{\pa\nu}=0,
	\qquad & x\in\pO, \ t\in (0,T_0), \\[1mm]
	\vp(x,T_0)=\vp_0(x), \quad \psi(x,T_0)=\psi_0(x),
	\qquad & x\in\Om.
	\ear
\ee 
Here the presence of the expressions $\na \ou^2$, and especially the ambition to control their action by means of 
diffusion-related contributions, motivates the inclusion of the conditions (\ref{b2}) and (\ref{b3}) which constitute
the main differences between the solution concept specified in the following definition and those that have been 
referred to in previous results on existence in (\ref{0}) (\cite{yuxiang_li_M3AS}, \cite{win_TRAN}, \cite{win_wu}).
\begin{defi}\label{db}
  Let $n\ge 1, \chi\ge 0$ and $\ell\ge 0$, and assume (\ref{i1}). 
  Then for $T\in (0,\infty]$, a pair $(u,v)$ of nonnegative functions defined a.e.~in $\Om\times (0,T)$
  will be called a {\em locally bounded weak solution} of (\ref{0}) in $\Om\times (0,T)$ if
  \be{b1}
	\lbal
	u\in L^\infty_{loc}(\bom\times [0,T))
	\qquad \mbox{and} \\[1mm]
	v\in L^\infty_{loc}([0,T);W^{1,\infty}(\Om))
	\ear
  \ee
  with
  \be{b2}
	\na u^2 \in L^2_{loc}(\bom\times [0,T);\R^n)
  \ee
  and
  \be{b3}
	\frac{1}{v} \in L^\infty_{loc}(\bom\times [0,T)),
  \ee
  and if the identities
  \bea{bu}
	- \int_0^T \io u\vp_t - \io u_0\vp(\cdot,0)
	&=& \frac{1}{2} \int_0^T \io u^2 v\Del\vp
	+ \frac{1}{2} \int_0^T \io u^2 \na v\cdot\na\vp \nn\\
	& & + \chi \int_0^T \io u^2 v\na v\cdot\na\vp
	+ \ell \int_0^T \io uv\vp
  \eea
  and
  \bea{bv}
	- \int_0^T \io v\vp_t - \io v_0\vp(\cdot,0)
	= \int_0^T \io v\Del\vp 
	- \int_0^T \io uv\vp
  \eea
  hold for all $\vp\in C_0^\infty(\bom\times [0,T))$ fulfilling $\frac{\pa\vp}{\pa\nu}=0$ on $\pO\times (0,T)$.
\end{defi}
{\bf Remark.} \quad
  As can readily be seen by a standard approximation argument, any locally bounded weak solution $(u,v)$ of (\ref{0}) in
  $\Om\times (0,T)$ satisfies (\ref{bu}) and (\ref{bv}) actually for each $\vp\in C^0(\bom\times [0,T))$ which
  has compact support in $\bom\times [0,T)$ and is such that its derivatives $\vp_t, \frac{\pa\vp}{\pa x_i}$ and
  $\frac{\pa^2\vp}{\pa x_i \pa x_j}$ belong to $L^\infty(\Om\times (0,T))$ for all $i\in\{1,...,n\}$ and $j\in\{1,...,n\}$.\abs
Based on an analysis of (\ref{01}), the first of our main results now asserts uniqueness in this framework. 
\begin{theo}\label{theo55}
  Let $n\ge 1$ and $\Om\subset\R^n$ be a bounded domain with smooth boundary, let $\chi\ge 0$ and $\ell\ge 0$,
  and suppose that
  \be{i1}
	\lbal
	u_0\in L^\infty(\Om)
	\mbox{ is nonnegative, \quad and that} \\[1mm]
	v_0 \in W^{1,\infty}(\Om)	
	\mbox{ is such that $v>0$ in $\bom$.}
	\ear
  \ee
  Then for each $T\in (0,\infty]$, the problem (\ref{0}) admits, in the sense of Definition \ref{db}, 
  at most one locally bounded weak solution in $\Om\times (0,T_0)$.
\end{theo}
{\bf Remark.} \quad
We emphasize that Theorem \ref{theo55} does not presuppse continuity of the initial data $u_0$, which is in sharp
contrast to the requirements made in related works on even simpler variants of (\ref{0}),
containing only single degeneracies of porous medium type, where H\"older continuity is explicitly presupposed
(\cite{sugiyama_kawakami_JDE}, \cite{ki_ahm_lee}).\Abs
{\bf Main results II: Continuous dependence on large timescales.} \quad
Our intention in the second part of this manuscript is to reveal on the basis of Theorem \ref{theo55} 
that (\ref{0}) enjoys a qualitative property which appears to have rarely been seen before 
in contexts of parabolic problems admitting multiple equilibria, where we note that (\ref{0}) actually possesses the
continuum of steady states given by $(u_\star,0)$ with arbitrary sufficiently regular $u_\star:\Om\to [0,\infty)$.
Namely, the following statement asserts that in low-dimensional convex domains and in a suitable topological framework, 
solutions depend continuously on the parameters $\chi$ and $\ell$, and especially also on the initial data,
not only on finite time intervals but actually on the whole half-line $(0,\infty)$:
\begin{theo}\label{theo33}
  Let $n\le 2$ and $\Om\subset\R^n$ be a bounded convex domain with smooth boundary, and let $K>0$.
  Then there exists $\Lam(K)>0$ with the property that if
  \be{chi_ell}
	\chi\in [0,K)
	\qquad \mbox{and} \quad
	\ell\in [0,K),
  \ee
  if
  \be{i2}
	\lbal
	u_0\in C^0(\bom)
	\mbox{ is nonnegative \quad and} \\[1mm]
	v_0 \in W^{1,\infty}(\Om)	
	\mbox{ is such that $v_0>0$ in $\bom$,}
	\ear
  \ee
  and if
  \be{33.92}
	\|u_0\|_{L^\infty(\Om)} < K,
	\quad
	\io u_0>\frac{1}{K},
	\quad
	\|v_0\|_{L^\infty(\Om)} < K
	\quad \mbox{and} \quad
	\Big\|\frac{\na v_0}{v_0}\Big\|_{L^\infty(\Om)} < K,
  \ee
  then (\ref{0}) admits a unique locally bounded weak solution 
  $(u,v)=(u^{(\chi,\ell,u_0,v_0)},v^{(\chi,\ell,u_0,v_0)})$ 
  in $\Om\times (0,\infty)$ which additionally satisfies
  \be{33.922}
	\lbal
	u\in C^0(\bom\times [0,\infty)) 
	\qquad \mbox{and} \\[1mm]
	v\in C^0(\bom\times [0,\infty)) \cap C^{2,1}(\bom\times (0,\infty))
	\ear
  \ee
  as well as
  \be{33.93}
	\|u\|_{L^\infty(\Om\times (0,\infty))} \le \Lam(K),
  \ee
  and this solution has the property that whenever
  $(\chi_k)_{k\in\N} \subset [0,\infty)$, $(\ell_k)_{k\in\N} \subset [0,\infty)$,
  $(u_{0k})_{k\in\N} \subset C^0(\bom)$ and $(v_{0k})_{k\in\N} \subset W^{1,\infty}(\Om)$ are such that for every
  $k\in\N$ we have $u_{0k}\not\equiv 0$ as well as $u_{0k}\ge 0$ and $v_{0k}>0$ in $\bom$, and that
  \be{33.777}
	\chi_k \to \chi
	\qquad \mbox{and} \qquad
	\ell_k\to\ell
  \ee
  as well as
  \be{33.2}
	u_{0k} \to u_0
	\quad \mbox{in } C^0(\bom)
	\qquad \mbox{and} \qquad
	v_{0k} \to v_0
	\quad \mbox{in } W^{1,\infty}(\Om)
  \ee
  as $k\to\infty$, the corresponding locally bounded weak solutions 
  $(u_k,v_k):=(u^{(\chi_k,\ell_k,u_{0k},v_{0k})},v^{(\chi_k,\ell_k,u_{0k},v_{0k})})$
  of (\ref{0}) in $\Om\times (0,\infty)$
  have the property that
  \be{33.3}
	\|u_k-u\|_{L^\infty(\Om\times (0,\infty))}
	+ \|v_k-v\|_{L^\infty(\Om\times (0,\infty))}
	\to 0
	\qquad \mbox{as } k\to\infty.
  \ee
\end{theo}
{\bf Remark.} \quad
For more regular initial data satisfying $0\not\equiv u_0\in W^{1,\infty}(\Om)$ and $v_0\in W^{1,\infty}(\Om)$ with $u_0\ge 0$
and $v_0>0$ in $\bom$, 
it has essentially been known from previous literature already that one can find a pair $(u,v)$ 
of nonnegative functions which solve (\ref{0}) in the sense that (\ref{b1}), (\ref{b3}), (\ref{bu}) and (\ref{bv}) hold.
As a by-product to the discovery of the above property of continuous dependence, Theorem \ref{theo55} 
extends this existence theory to solutions additionally fulfilling (\ref{b2}), and to initial data which with regard to the first component are only continuous.\abs
The following second result in this direction addresses more general domains and especially also three-dimensional
scenarios, and ensures that a similar property can then be observed at least in the physically most relevant situation 
when $v_0$ additionally satisfies a suitable smallness condition. 
The following statement actually is more comprehensive by drawing a conclusion of this type also for 
initial data of arbitrary size, provided that $\chi$ and $\ell$ are sufficiently small:
\begin{theo}\label{theo34}
  Let $n\le 3$ and $\Om\subset\R^n$ be an arbitrary bounded domain with smooth boundary.
  Then there exist $(\del(K,\LL))_{K>0,\LL>0} \subset (0,\infty)$ and $(\Lam(K,\LL))_{K>0,\LL>0} \subset (0,\infty)$
  with the property that
  \be{34.01}
	\del(K,\LL) \to + \infty
	\quad \mbox{for all $K>0$ as $\LL\searrow 0$}
  \ee
  and that whenever $\chi$, $\ell$, $u_0$ and $v_0$ are such that
  \bas
	\chi\in [0,\LL)
	\qquad \mbox{and} \qquad
	\ell\in [0,\LL),
  \eas
  and that (\ref{i2}) holds with
  \be{34.1}
	\|u_0\|_{L^\infty(\Om)} < K,
	\quad \io u_0 > \frac{1}{K},
	\quad \|v_0\|_{L^\infty(\Om)} < K
	\quad \mbox{and} \quad
	\Big\| \frac{\na v_0}{v_0}\Big\|_{L^\infty(\Om)} < K
  \ee
  as well as
  \be{34.2}
	\io v_0 < \del(K,\LL),
  \ee
  the problem (\ref{0}) possesses a unique locally bounded weak solution 
  $(u,v)=(u^{(\chi,\ell,u_0,v_0)},v^{(\chi,\ell,u_0,v_0)})$ in $\Om\times (0,\infty)$ fulfilling (\ref{33.922}) and
  \be{34.93}
	\|u\|_{L^\infty(\Om\times (0,\infty))} \le \Lam(K,\LL).
  \ee
  Moreover, if $(\chi_k)_{k\in\N} \subset [0,\infty)$, $(\ell_k)_{k\in\N} \subset [0,\infty)$,
  $(u_{0k})_{k\in\N} \subset C^0(\bom)$ and $(v_{0k})_{k\in\N} \subset W^{1,\infty}(\Om)$ are such that
  $u_{0k} \not\equiv 0$, $u_{0k} \ge 0$ and $v_{0k}>0$ in $\bom$ for all $k\in\N$, and that (\ref{33.777}) and (\ref{33.2}) hold,
  then one can find $k_0\in\N$ such that for each $k>k_0$ there exists a unique locally bounded weak solution 
  $(u_k,v_k)=(u^{(\chi_k,\ell_k,u_{0k},v_{0k})},v^{(\chi_k,\ell_k,u_{0k},v_{0k})})$
  of (\ref{0}) in $\Om\times (0,\infty)$ 
  with parameters $\chi_k$ and $\ell_k$ and
  with initial data $(u_{0k},v_{0k})$, additionally satisfying
  $u_k\in C^0(\bom\times [0,\infty))$ and $v_k\in C^0(\bom\times [0,\infty)) \cap C^{2,1}(\bom\times (0,\infty))$,
  and the sequence of these solutions has the property in (\ref{33.3}).
\end{theo}
\mysection{Uniqueness. Proof of Theorem \ref{theo55}}\label{sect2}
To initiate our analysis toward Theorem \ref{theo55}, given two locally bounded weak solutions 
let us perform an essentially straightforward testing procedure to derive an integral identity satisfied by corresponding
differences.
\begin{lem}\label{lem51}
  Let $n\ge 1, \chi\ge 0$ and $\ell\ge 0$, assume (\ref{i1}), and suppose that for some $T\in (0,\infty]$, the pairs
  $(u,v)$ and $(\ou,\ovv)$ both are locally bounded weak solutions of (\ref{0}) in $\Om\times (0,T)$.
  Then given any $\vp_0\in C_0^\infty(\Om)$ and $\psi_0\in C_0^\infty(\Om)$, one can find a null set
  $N=N(u,v,\ou,\ovv,\vp_0,\psi_0)\subset (0,T)$ with the property that whenever $T_0\in (0,T)\sm N$ and
  $\vp\in C^{2,1}(\bom\times [0,T_0])$ as well as $\psi\in C^{2,1}(\bom\times [0,T_0])$ are such that
  $\vp(\cdot,T_0)=\vp_0$ and $\psi(\cdot,T_0)=\psi_0$ in $\Om$, it follows that
  \bea{51.2}
	& & \hs{-18mm}
	\io w(\cdot,T_0)\vp_0
	+ \io z(\cdot,T_0) \psi_0 \nn\\
	&=& \int_0^{T_0} \io w\cdot\Big\{ \vp_t + \frac{1}{2}(u+\ou) v\Del\vp 
	+ \frac{1}{2}(u+\ou) \na v\cdot\na\vp 
	+ \chi (u+\ou) v\na v\cdot\na\vp 
	+ \ell v\vp
	- v\psi \Big\} \nn\\
	& & + \int_0^{T_0} \io z\cdot \Big\{ 
	\psi_t + \Del\psi 
	- \chi \ou^2 \ovv\Del\vp
	- \frac{1}{2} \na\ou^2 \cdot\na\vp 
	+ \chi \ou^2 \na v\cdot\na \vp \nn\\
	& & \hs{24mm}
	- \chi \ovv \na\ou^2 \cdot\na\vp
	- \chi\ou^2 \na\ovv\cdot\na\vp 
	+ \ell \ou \vp - \ou\psi \Big\},
  \eea
  where
  \be{wz}
	w:=u-\ou
	\qquad \mbox{and} \qquad
	z:=v-\ovv.
  \ee
\end{lem}
\proof
  Let $N=N(u,v,\ou,\ovv,\vp_0,\psi_0)\subset (0,T)$ be a null set such that any $T_0\in (0,T)\sm N$ is a common Lebesgue point
  of $(0,T)\ni t\mapsto \io w(\cdot,t)$ and $(0,T)\ni t\mapsto \io z(\cdot,t)$.
  Then fixing $T_0\in (0,T)\sm N$ as well as $\vp\in C^{2,1}(\bom\times [0,T_0])$ and $\psi\in C^{2,1}(\bom\times [0,T_0])$
  such that $\vp(\cdot,T_0)=\vp_0$ and $\psi(\cdot,T_0)=\psi_0$, using these Lebesgue point features we see that
  \be{51.02}
	\frac{1}{\del} \int_{T_0-\del}^{T_0} \io w\vp
	\to \io w(\cdot,T_0)\vp_0
	\quad \mbox{and} \quad
	\frac{1}{\del} \int_{T_0-\del}^{T_0} \io z\psi
	\to \io z(\cdot,T_0)\psi_0
	\qquad \mbox{as } (0,T_0)\ni\del\to 0,
  \ee
  because
  \bas
	\bigg| \frac{1}{\del} \int_{T_0-\del}^{T_0} \io w\vp
	- \frac{1}{\del} \int_{T_0-\del}^{T_0} \io w\vp_0 \bigg|
	&=& \bigg| \frac{1}{\del} \int_{T_0-\del}^{T_0} \io w(x,t) \cdot\big\{ \vp(x,t)-\vp_0(x)\big\} dxdt \bigg| \\
	&\le& \Big\{ {\rm{ess}} \sup_{\hs{-3mm} t\in (0,T_0)} \|w(\cdot,t)\|_{L^1(\Om)} \Big\} \cdot
	\sup_{t\in (T_0-\del,T_0)} \|\vp(\cdot,t)-\vp_0\|_{L^\infty(\Om)} \\[2mm]
	&\to& 0
	\qquad \mbox{as } (0,T_0)\ni\del\to 0
  \eas
  by continuity of $\vp$, and because a similar argument applies to the expressions in the second relation in (\ref{51.02}).\\
  Next choosing $\zeta\in W^{1,\infty}(\R)$ to be the piecewise linear function determined by the properties
  $\zeta\equiv 1$ on $(-\infty,-1]$, $\zeta\equiv 0$ on $[0,\infty)$ and $\zeta'\equiv -1$ on $(-1,0)$, given $T_0\in (0,T)\sm N$
  and $\del\in (0,T_0)$ we let $\zd(t):=\zeta(\frac{t-T_0}{\del})$, $t\ge 0$, as well as
  \bas
	\vp_\del(x,t):=\zd(t) \vp(x,t)
	\quad \mbox{and} \quad
	\psi_\del(x,t):=\zd(t)\psi(x,t),
	\qquad (x,t)\in \bom\times [0,T_0].
  \eas
  Then according to the remark following Definition \ref{db}, we may apply (\ref{bu}) to the test function $\vp_\del$ to find that
  \bas
	- \int_0^{T_0} \io u\vp_{\del t} - \io u_0 \vp_\del(\cdot,0)
	&=& \frac{1}{2} \int_0^{T_0} \io u^2 v\Del\vp_\del
	+ \frac{1}{2} \int_0^{T_0} \io u^2 \na v\cdot\na\vp_\del \nn\\
	& & + \chi \int_0^{T_0} \io u^2 v\na v\cdot\na\vp_\del
	+ \ell \int_0^{T_0} \io uv\vp_\del,
  \eas
  which combined with a corresponding identity satisfied by $(\ou,\ovv)$ implies that if we let $w$ and $z$ be as in
  (\ref{wz}), then since $\varphi_{\del t}=\zd(t) \vp_t + \zd'(t) \vp$, and since $u^2-\ou^2=(u+\ou)w$,
  \bea{51.4}
	& & \hs{-13mm}
	- \int_0^{T_0} \io \zd(t) w\vp_t 
	+ \frac{1}{\del} \int_{T_0-\del}^{T_0} \io w\vp \nn\\
	&=& \frac{1}{2} \int_0^{T_0} \io  (u^2 v - \ou^2\ovv) \Del\vp_\del
	+ \frac{1}{2} \int_0^{T_0} \io (u^2\na v - \ou^2 \na\ovv) \cdot\na\vp_\del \nn\\
	& & + \chi \int_0^{T_0} \io (u^2v\na v - \ou^2\ovv\na\ovv) \cdot\na\vp_\del
	+ \ell \int_0^{T_0} \io (uv-\ou\ovv) \vp_\del \nn\\
	&=& \frac{1}{2} \int_0^{T_0} \io (u^2-\ou^2)v\Del\vp_\del
	+ \frac{1}{2} \int_0^{T_0} \io \ou^2 (v-\ovv) \Del\vp_\del \nn\\
	& & + \frac{1}{2} \int_0^{T_0} \io (u^2-\ou^2) \na v\cdot\na\vp_\del
	+ \frac{1}{2} \int_0^{T_0} \io \ou^2 \na (v-\ovv) \cdot\na\vp_\del \nn\\
	& & + \chi \int_0^{T_0} \io (u^2-\ou^2) v\na v\cdot\na\vp_\del
	+ \chi \int_0^{T_0} \io \ou^2 (v-\ovv)\na v\cdot\na\vp_\del
	+ \chi \int_0^{T_0} \io \ou^2 \ovv \na(v-\ovv) \cdot\na\vp_\del \nn\\
	& & + \ell \int_0^{T_0} \io (u-\ou)v\vp_\del
	+ \ell \int_0^{T_0} \io \ou (v-\ovv)\vp_\del \nn\\
	&=& \int_0^{T_0} \io \zd(t) w\cdot\Big\{ \frac{1}{2} (u+\ou) v\Del\vp
	+ \frac{1}{2}(u+\ou)\na v\cdot\na\vp 
	+ \chi(u+\ou) v\na v\cdot\na\vp
	+ \ell v\vp \Big\} \nn\\
	& & + \int_0^{T_0} \io \zd(t) z\cdot\Big\{ 
	\frac{1}{2} \ou^2 \Del\vp
	+ \chi \ou^2 \na v\cdot\na\vp
	+ \ell \ou\vp \Big\} \nn\\
	& & + \int_0^{T_0} \io \zd(t) \na z\cdot \Big\{
	\frac{1}{2} \ou^2 \na\vp
	+ \chi \ou^2 \ovv \na\vp\Big\}
	\qquad \mbox{for all } \del\in (0,T_0).
  \eea
  Here, using that Definition \ref{db} ensures the inclusion $\na \ou^2\in L^2(\Om\times (0,T_0))$ we may integrate by parts
  to see that for all $\del\in (0,T_0)$,
  \bas
	& & \hs{-20mm}
	\int_0^{T_0} \io \zd(t)\na z\cdot\Big\{ \frac{1}{2} \ou^2 \na\vp + \chi \ou^2 \ovv \na\vp\Big\} \nn\\
	&=& \int_0^{T_0} \io \zd(t) z\cdot\Big\{ 
	- \frac{1}{2} \ou^2 \Del\vp
	- \frac{1}{2} \na\ou^2\cdot\na\vp
	- \chi \ou^2 \ovv \Del\vp
	- \chi \ovv \na\ou^2 \cdot\na\vp
	- \chi \ou^2 \na\ovv\cdot\na\vp \Big\},
  \eas
  whence again explicitly using that $u,\ou, v$ and $\na v$ are bounded in $\Om\times (0,T_0)$, 
  from (\ref{51.4}) we infer by letting $\del\searrow 0$ and recalling (\ref{51.02}) that
  \bea{51.5}
	& & \hs{-20mm}
	- \int_0^{T_0} \io w\vp_t
	+ \io w(\cdot,T_0)\vp_0 \nn\\
	&=& \int_0^{T_0} \io w\cdot\Big\{ \frac{1}{2} (u+\ou)v\Del\vp 
	+ \frac{1}{2} (u+\ou) \na v\cdot\na\vp
	+ \chi (u+\ou) v\na v\cdot\na\vp
	+ \ell v\vp \big\} \nn\\
	& & + \int_0^{T_0} \io z\cdot\Big\{
	\frac{1}{2} \ou^2 \Del\vp
	+ \chi \ou^2 \na v\cdot\na\vp
	+ \ell \ou \vp
	- \frac{1}{2} \ou^2 \Del\vp \nn\\
	& & \hs{20mm}
	- \frac{1}{2} \na\ou^2 \cdot\na\vp
	- \chi \ou^2 \ovv \Del\vp
	- \chi \ovv \na\ou^2\cdot\na\vp
	- \chi \ou^2 \na \ovv\cdot\na\vp \Big\}.
  \eea
  Similarly, based on the fact that in view of the remark following Definition \ref{db} we have
  \bas
	- \int_0^{T_0} \io v\psi_{\del t} 
	- \io v_0 \psi_\del(\cdot,0)
	= \int_0^{T_0} \io v\Del\psi_\del
	- \int_0^{T_0} \io uv\psi_\del
  \eas
  as well as a corresponding relation involving $(\ou,\ovv)$, upon taking differences we obtain that
  \bas
	& & \hs{-20mm}
	- \int_0^{T_0} \io \zd(t) z\psi_t
	+ \frac{1}{\del} \int_{T_0-\del}^{T_0} \io z\psi \nn\\
	&=& \int_0^{T_0} \io \zd(t) z\Del\psi
	- \int_0^{T_0} \io \zd(t) wv\psi
	- \int_0^{T_0} \io \zd(t) \ou z\psi
	\qquad \mbox{for all } \del\in (0,T_0),
  \eas
  which in the limit $\del\searrow 0$ implies that, by (\ref{51.02}),
  \bas		
	- \int_0^{T_0} \io z\psi_t
	+ \io z(\cdot,T_0)\psi_0 
	= \int_0^{T_0} \io z\cdot \big\{ \Del\psi - \ou\psi\big\}
	+ \int_0^{T_0} \io w\cdot \big\{ -v\psi\big\}.
  \eas
  Adding this to (\ref{51.5}) yields (\ref{51.2}) after straightforward rearrangement.
\qed
While for sufficiently smooth and uniformly positive solutions $(u,v)$ and $(\ou,\ovv)$ the above lemma 
suggests to use function pairs solving (\ref{01}) as test functions in (\ref{51.2}), under the assumptions 
in Theorem \ref{theo55} the problem (\ref{0}) cannot be expected to admit solutions with these properties.
Our key step in Lemma \ref{lem54} will therefore apply (\ref{51.2}) to test functions which solve a non-degenerate
regularized variant of (\ref{01}), the ingredients of which will be constructed in the following.
\begin{lem}\label{lem52}
  Let $n\ge 1, \chi\ge 0$ and $\ell\ge 0$, assume (\ref{i1}), and suppose that $(u,v)$ is a locally bounded weak solution
  of (\ref{0}) for some $T\in (0,\infty]$.
  Then given $T_0\in (0,T)$, one can fix $K=K(u,v,T_0)>0$ as well as $(U_j[u,v,T_0])_{j\in\N} \subset C^\infty(\bom\times [0,T_0])$
  and $(V_j[u,v,T_0])_{j\in\N} \subset C^\infty(\bom\times [0,T_0])$ such that writing $u_j:=U_j[u,v,T_0]$ and $v_j:=V_j[u,v,T_0]$
  for $j\in\N$, we have
  \be{52.1}
	0 \le u_j \le K
	\quad \mbox{in } \Om\times (0,T_0)
	\qquad \mbox{for all } j\in\N
  \ee
  and
  \be{52.2}
	\frac{1}{K} \le v_j \le K
	\quad \mbox{in } \Om\times (0,T_0)
	\qquad \mbox{for all } j\in\N
  \ee
  as well as
  \be{52.4}
	|\na v_j| \le K
	\quad \mbox{in } \Om\times (0,T_0)
	\qquad \mbox{for all } j\in\N,
  \ee
  and that as $j\to\infty$ we have
  \be{52.5}
	u_j\to u
	\quad \mbox{and} \quad
	v_j\to v
	\qquad \mbox{a.e.~in } \Om\times (0,T_0)
  \ee
  and
  \be{52.6}
	j \int_0^{T_0} \io (u_j-u)^2 \to 0
  \ee
  and
  \be{52.7}
	j \int_0^{T_0} \io (v_j-v)^2 \to 0
  \ee
  as well as
  \be{52.8}
	\na u_j^2 \to \na u^2
	\qquad \mbox{in } L^2(\Om\times (0,T_0))
  \ee
  and
  \be{52.9}
	\na v_j \to \na v
	\qquad \mbox{in } L^2(\Om\times (0,T_0)).
  \ee
\end{lem}
\proof
  We let $h:=u^2$ and set
  \bas
	h^{(\xi)}(x,t):=\big[ e^{\xi\Del} h(\cdot,t)\big](x)
	\qquad \mbox{for $x\in\bom, t\in [0,T_0]$ and $\xi\in (0,1)$,}
  \eas
  where $(e^{\xi\Del})_{\xi\ge 0}$ denotes the Neumann heat semigroup on $\Om$.
  Then standard smoothing properties thereof ensure that $h^{(\xi)}\in C^\infty(\bom\times [0,T_0])$ for all $\xi\in (0,1)$,
  and a comparison principle asserts that for each $\xi\in (0,1)$ we have $h^{(\xi)}\ge 0$ as well as
  \be{52.10}
	h^{(\xi)}(x,t) \le \|h(\cdot,t)\|_{L^\infty(\Om)} \le c_1\equiv c_1(u,v,T_0):=\|u\|_{L^\infty(\Om\times (0,T_0))}^2
	\quad \mbox{in $\Om$} \qquad \mbox{for all } t\in [0,T_0].
  \ee
  Known boundedness and continuity features of $(e^{\xi\Del})_{\xi\ge 0}$ furthermore guarantee that
  \be{52.11}
	\io \big|\na h^{(\xi)}(\cdot,t)\big|^2
	\le \io |\na h(\cdot,t)|^2
	\qquad \mbox{for all } t\in (0,T_0),
  \ee
  and that
  \be{52.12}
	h^{(\xi)}(\cdot,t) \to h(\cdot,t)
	\ \mbox{in } L^1(\Om)
	\quad \mbox{and} \quad
	\na h^{(\xi)}(\cdot,t) \to \na h(\cdot,t)
	\ \mbox{in } L^2(\Om)
	\quad \mbox{for a.e.~} t\in (0,T_0)
	\qquad \mbox{as } \xi\searrow 0,
  \ee
  so that, in particular,
  \bas
	& & \hs{-30mm}
	\int_0^{T_0} \io |h^{(\xi)}-h|
	+ \int_0^{T_0} \io |\na h^{(\xi)}-\na h|^2 \\
	&=& \int_0^{T_0} \|h^{(\xi)}(\cdot,t)-h(\cdot,t)\|_{L^1(\Om)} dt
	+ \int_0^{T_0} \|\na h^{(\xi)}(\cdot,t)-\na h(\cdot,t)\|_{L^2(\Om)}^2 dt \\
	&\to& 0
	\qquad \mbox{as } j\to\infty
  \eas
  due to the dominated convergence theorem when combined with (\ref{52.10}), (\ref{52.11}) and (\ref{52.12}).
  We can therefore find $(\xi_j)_{j\in\N} \subset (0,1)$ such that $\xi_j\to 0$ as $j\to\infty$, that
  \be{52.13}
	j\int_0^{T_0} \io |h^{(\xi_j)}-h| \to 0
	\quad \mbox{and} \quad
	\int_0^{T_0} \io |\na h^{(\xi_j)}-\na h|^2 \to 0
	\qquad \mbox{as } j\to\infty,
  \ee
  and that
  \be{52.14}
	h^{(\xi_j)} \to h
	\quad \mbox{a.e.~in } \Om\times (0,T_0)
	\qquad \mbox{as } j\to\infty.
  \ee
  Letting 
  \bas
	u_j:=\sqrt{h^{(\xi_j)}+\frac{1}{j^2}}
	\qquad \mbox{for } j\in\N,
  \eas
  we thus obtain a sequence of functions $u_j$ which lie in $C^\infty(\bom\times [0,T_0])$ by smoothness of 
  $0\le X \mapsto \sqrt{X+\frac{1}{j^2}}$, which thanks to (\ref{52.10}) satisfy $0\le u_j \le \sqrt{c_1+1}$
  in $\Om\times (0,T_0)$ for all $j\in\N$, which due to (\ref{52.14}) and (\ref{52.13}) have the property that
  $u_j\to u$ a.e.~in $\Om\times (0,T_0)$ as $j\to\infty$, that since $|\sqrt{X}-\sqrt{Y}|^2 \le |X-Y|$ for all
  $X\ge 0$ and $Y\ge 0$ we have
  \bas
	j \int_0^{T_0} \io (u_j-u)^2
	&=& j \int_0^{T_0} \io \Big\{ \sqrt{h^{(\xi_j)}+\frac{1}{j^2}} - \sqrt{h}\Big\}^2 \\
	&\le& j \int_0^{T_0} \io \big| h^{(\xi_j)}+\frac{1}{j^2} - h\Big| \\
	&\le& j \int_0^{T_0} \io \big|h^{(\xi_j)}-h\big|
	+ \frac{|\Om| T_0}{j} \\[2mm]
	&\to& 0
	\qquad \mbox{as } j\to\infty,
  \eas
  and for which from (\ref{52.13}) it moreover follows that
  \bas
	\int_0^{T_0} \io |\na u_j^2 - \na u^2|^2
	&=& \int_0^{T_0} \io \Big| \na \Big( h^{(\xi_j)}+\frac{1}{j^2}\Big)-\na h\Big|^2 \\
	&=& \int_0^{T_0} \io |\na h^{(\xi_j)}-\na h|^2 \\[2mm]
	&\to& 0
	\qquad \mbox{as } j\to\infty.
  \eas
  A sequence $(v_j)_{j\in\N} \subset C^\infty(\bom\times [0,T_0])$ with the respective properties claimed in 
  (\ref{52.2}), (\ref{52.4}), (\ref{52.5}), (\ref{52.6}) and (\ref{52.9}) can be obtained by a simplified variant
  of this construction,
  based on the use of the functions $g^{(\xi)}(x,t):=[e^{\xi\Del}v(\cdot,t)](x)$, $(x,t)\in\bom\times [0,T_0]$, $\xi\in (0,1)$,
  and the fact that then $\inf_{\Om\times (0,T_0)} v \le g^{(\xi)} \le \|v\|_{L^\infty(\Om\times (0,T_0))}$
  in $\Om\times (0,T_0)$ for all $\xi\in (0,1)$ according to two comparison arguments.
\qed
We can now introduce the announced non-degenerate approximations of (\ref{01}), and provide some information on regularity
of its solutions.
While the following statement in this regard already does rely on the pointwise bounds in (\ref{52.1})-(\ref{52.4}),
properties of population gradients such as that in (\ref{52.8}) are not yet needed here.
\begin{lem}\label{lem53}
  Let $n\ge 1, \chi\ge 0$ and $\ell\ge 0$, let $u_0$ and $v_0$ be such that (\ref{i1}) holds, let $(u,v)$ 
  and $(\ou,\ovv)$ be locally bounded weak solutions of (\ref{0}) in $\Om\times (0,T)$ for some $T\in (0,\infty]$,
  let $\vp_0\in C_0^\infty(\Om)$ and $\psi_0\in C_0^\infty(\Om)$,
  and taking $N=N(u,v,\ou,\ovv,\vp_0,\psi_0)$ from Lemma \ref{lem51}, for fixed $T_0\in (0,T)\sm N$ and $j\in\N$, let
  $u_j:=U_j[u,v,T_0]$, $v_j:=V_j[u,v,T_0]$, $\ou_j:=U_j[\ou,\ovv,T_0]$ and $\ovv_j:=V_j[\ou,\ovv,T_0]$ be as provided
  by Lemma \ref{lem52}.
  Then for each $j\in\N$, the problem
  \be{53.1}
	\lball
	\pa_t \vp_j
	= - \frac{1}{2} (u_j+\ou_j+\frac{1}{j})v_j\Del\vp_j
	- \frac{1}{2}(u_j+\ou_j)\na v_j\cdot\na\vp_j \\
	\hs{14mm}
	- \chi(u_j+\ou_j) v_j \na v_j\cdot\na\vp_j 
	- \ell v_j \vp_j
	+ v_j \psi_j,
	\qquad & x\in \Om, \ t\in (0,T_0), \\[1mm]
	\pa_t \psi_j 
	= -\Del\psi_j
	+ \chi \ou_j^2 \ovv_j \Del\vp_j
	+ \frac{1}{2} \na \ou_j^2 \cdot\na\vp_j \\
	\hs{14mm}
	- \chi \ou_j^2 \na v_j\cdot\na\vp_j
	+ \chi \ovv_j \na \ou_j^2 \cdot\na\vp_j
	+ \chi \ou_j^2 \na\ovv_j\cdot\na\vp_j \\
	\hs{14mm}
	- \ell \ou_j \vp_j
	+ \ou_j \psi_j,
	\qquad & x\in \Om, \ t\in (0,T_0), \\[1mm]
	\frac{\pa\vp_j}{\pa\nu}=\frac{\pa\psi_j}{\pa\nu}=0,
	\qquad & x\in\pO, \ t\in (0,T_0), \\[1mm]
	\vp_j(x,T_0)=\vp_0(x), \quad \psi_j(x,T_0)=\psi_0(x),
	\qquad & x\in\Om,
	\ear
  \ee 
  admits a classical solution $(\vp_j,\psi_j)$ with
  \be{53.2}
	\vp_j\in C^\infty(\bom\times [0,T_0])
	\qquad \mbox{and} \qquad
	\psi_j\in C^\infty(\bom\times [0,T_0]).
  \ee
  The family of these solutions has the property that with some $\Gamma=\Gamma(u,v,\ou,\ovv,T_0,\vp_0,\psi_0)>0$ we have
  \be{53.3}
	\io \vp_j^2(\cdot,t)
	+ \io |\na \vp_j(\cdot,t)|^2
	\le \Gamma
	\qquad \mbox{for all $t\in (0,T_0)$ and } j\in\N
  \ee
  and
  \be{53.4}
	\io \psi_j^2(\cdot,t)
	\le \Gamma
	\qquad \mbox{for all $t\in (0,T_0)$ and } j\in\N
  \ee
  as well as
  \be{53.6}
	\int_0^{T_0} \io |\Del\vp_j|^2 \le \Gamma j
	\qquad \mbox{for all } j\in\N.
  \ee
\end{lem}
\proof
  For each $j\in\N$, substituting $s=T_0-t$ for $t\in [0,T_0]$ we see that thanks to the nonnegativity of $u_j$ and $\ou_j$,
  (\ref{53.1}) actually is a non-degenerate parabolic problem which is linear with smooth coefficients, and which is of triangular
  cross-diffusive character. 
  As the inclusions $\vp_0\in C_0^\infty(\Om)$ and $\psi_0\in C_0^\infty(\Om)$ ensure that corresponding compatibility conditions
  of arbitrary order are fulfilled, standard parabolic theory (\cite{amann}) therefore ensures the existence of functions
  $\vp_j$ and $\psi_j$ which satisfy (\ref{53.2}) and solve (\ref{53.1}) in the classical sense in $\Om\times (0,T_0)$.\abs
  To derive (\ref{53.3})-(\ref{53.6}), we employ Lemma \ref{lem52} and hence find positive constants $c_i=c_i(u,v,\ou,\ovv,T_0)$,
  $i\in\{1,...,7\}$, such that
  \be{53.7}
	u_j \le c_1,
	\quad \ou_j \le c_2,
	\quad
	c_3 \le v_j \le c_4,
	\quad \ovv_j \le c_5,
	\quad |\na v_j| \le c_6
	\quad \mbox{and} \quad
	|\na\ovv_j| \le c_7
  \ee
  in $\Om\times (0,T_0)$ for all $j\in\N$.
  For $j\in\N$, we now test the first equation in (\ref{53.1}) by $\Del\vp_j$ to see that
  \bea{53.9}
	& & \hs{-20mm}
	- \frac{1}{2} \frac{d}{dt} \io |\na\vp_j|^2
	+ \frac{1}{2} \io (u_j+\ou_j)v_j |\Del\vp_j|^2
	+ \frac{1}{2j} \io v_j |\Del\vp_j|^2 \nn\\
	&=& - \frac{1}{2} \io (u_j+\ou_j) (\na v_j\cdot\na\vp_j) \Del\vp_j
	- \chi \io (u_j+\ou_j) v_j (\na v_j\cdot\na\vp_j)\Del\vp_j \nn\\
	& & - \ell \io v_j \vp_j\Del\vp_j
	+ \io v_j\psi_j \Del\vp_j
	\qquad \mbox{for all } t\in (0,T_0).
  \eea
  Here, Young's inequality together with (\ref{53.7}) entails that
  for all $t\in (0,T_0)$,
  \bea{53.10}
	- \frac{1}{2} \io (u_j+\ou_j) (\na v_j\cdot\na\vp_j) \Del\vp_j
	&\le& \frac{1}{4} \io (u_j+\ou_j)v_j |\Del\vp_j|^2
	+ \frac{1}{4} \io \frac{u_j+\ou_j}{v_j} |\na v_j|^2 |\na\vp_j|^2 \nn\\
	&\le& \frac{1}{4} \io (u_j+\ou_j)v_j |\Del\vp_j|^2
	+ \frac{(c_1+c_2)c_6^2}{4c_3} \io |\na\vp_j|^2,
  \eea
  and that
  \bea{53.11}
	- \chi \io (u_j+\ou_j) v_j (\na v_j\cdot\na\vp_j)\Del\vp_j
	&\le& \frac{1}{8} \io (u_j+\ou_j)v_j |\Del\vp_j|^2
	+ 2\chi^2 \io (u_j+\ou_j) v_j |\na v_j|^2 |\na\vp_j|^2 \nn\\
	&\le& \frac{1}{8} \io (u_j+\ou_j)v_j |\Del\vp_j|^2
	+ 2(c_1+c_2) c_4 c_6^2 \chi^2 \io |\na\vp_j|^2
  \eea
  for all $t\in (0,T_0)$,
  while two integrations by parts, followed by applications of Young's inequality and (\ref{53.7}), show that
  \bea{53.12}
	- \ell \io v_j\vp_j \Del\vp_j
	&=& \ell \io v_j |\na\vp_j|^2
	+ \ell \io \vp_j \na v_j\cdot\na\vp_j \nn\\
	&\le& c_4 \ell \io |\na\vp_j|^2
	+ \io |\na\vp_j|^2
	+ \frac{c_6^2 \ell^2}{4} \io \vp_j^2
	\qquad \mbox{for all } t\in (0,T_0)
  \eea
  and 
  \bea{53.13}
	\hs{-8mm}
	\io v_j \psi_j \Del\vp_j
	&=& - \io v_j \na\vp_j\cdot\na\psi_j
	- \io \psi_j \na v_j\cdot\na\vp_j \nn\\
	&\le& \frac{1}{2} \io |\na\psi_j|^2
	+ \frac{c_4^2}{2} \io |\na\vp_j|^2
	+ \io |\na\vp_j|^2
	+ \frac{c_6^2}{4} \io \psi_j^2
	\qquad \mbox{for all } t\in (0,T_0).
  \eea
  When inserted into (\ref{53.9}), the inequalities in (\ref{53.10})-(\ref{53.13}) imply that if we let
  $c_8\equiv c_8(u,v,\ou,\ovv,T_0):=\frac{(c_1+c_2)c_6^2}{4c_3} + 2(c_1+c_2) c_4 c_6^2 \chi^2 + c_4 \ell + \frac{c_4^2}{2} +2$,
  $c_9\equiv c_9(u,v,\ou,\ovv,T_0):=\frac{c_6^2 \ell^2}{4}$ and
  $c_{10}\equiv c_{10}(u,v,\ou,\ovv,T_0):=\frac{c_6^2}{4}$, then
  \bea{53.14}
	& & \hs{-30mm}
	- \frac{1}{2} \frac{d}{dt} \io |\na\vp_j|^2
	+ \frac{1}{8} \io (u_j+\ou_j)v_j |\Del\vp_j|^2
	+ \frac{1}{2j} \io v_j |\Del\vp_j|^2 \nn\\
	&\le& c_8 \io |\na\vp_j|^2 
	+ c_9 \io \vp_j^2
	+ \frac{1}{2} \io |\na \psi_j|^2
	+ c_{10} \io \psi_j^2
	\qquad \mbox{for all } t\in (0,T_0).
  \eea
  To compensate the third to last summand herein, we multiply the first equation in (\ref{53.1}) by $-\vp_j$ and thereby obtain that
  for all $t\in (0,T_0)$,
  \bea{53.15}
	- \frac{1}{2} \frac{d}{dt} \io \vp_j^2
	&=& \frac{1}{2} \io \Big(u_j+\ou_j+\frac{1}{j}\Big)v_j\vp_j\Del\vp_j
	+ \frac{1}{2} \io (u_j+\ou_j) \vp_j \na v_j\cdot\na\vp_j  \nn\\
	& & + \chi \io (u_j+\ou_j) v_j \vp_j \na v_j\cdot\na\vp_j 
	+ \ell \io v_j \vp_j^2
	- \io v_j \vp_j \psi_j,
  \eea
  where again by Young's inequality and (\ref{53.7}),
  \bas
	\frac{1}{2} \io \Big(u_j+\ou_j+\frac{1}{j}\Big)v_j\vp_j\Del\vp_j
	&\le& \frac{1}{16} \io (u_j+\ou_j) v_j |\Del\vp_j|^2
	+ \io (u_j+\ou_j) v_j \vp_j^2 \nn\\
	& & + \frac{1}{4j} \io v_j |\Del\vp_j|^2
	+ \frac{1}{4j} \io v_j \vp_j^2 \nn\\
	&\le& \frac{1}{16} \io (u_j+\ou_j) v_j |\Del\vp_j|^2
	+ (c_1+c_2) c_4 \io \vp_j^2 \nn\\
	& & + \frac{1}{4j} \io v_j |\Del\vp_j|^2
	+ \frac{c_4}{4} \io \vp_j^2
  \eas
  and
  \bas	
	\frac{1}{2} \io (u_j+\ou_j) \vp_j \na v_j\cdot\na\vp_j
	\le \io |\na\vp_j|^2
	+ \frac{(c_1+c_2)^2 c_6^2}{16} \io\vp_j^2 
  \eas
  and
  \bas
	\chi \io (u_j+\ou_j) v_j \vp_j \na v_j\cdot\na\vp_j
	\le \io |\na\vp_j|^2
	+ \frac{(c_1+c_2)^2 c_4^2 c_6^2 \chi^2}{4} \io \vp_j^2
  \eas
  as well as
  \bas
	\ell \io v_j \vp_j^2
	\le c_4 \ell \io \vp_j^2
  \eas
  and
  \bas
	- \io v_j \vp_j \psi_j
	\le \io \vp_j^2
	+ \frac{c_4^2}{4} \io \psi_j^2
  \eas
  for all $t\in (0,T_0)$.
  Consequently,
  \bea{53.16}
	- \frac{1}{2} \frac{d}{dt} \io \vp_j^2
	&\le& \frac{1}{16} \io (u_j+\ou_j) v_j |\Del\vp_j|^2
	+ \frac{1}{4j} \io v_j |\Del\vp_j|^2 \nn\\
	& & + 2 \io |\na\vp_j|^2
	+ c_{11} \io \vp_j^2
	+ c_{12} \io \psi_j^2
	\qquad \mbox{for all } t\in (0,T_0)
  \eea
  with
  $c_{11}\equiv c_{11}(u,v,\ou,\ovv,T_0):=(c_1+c_2)c_4 + \frac{c_4}{4} + \frac{(c_1+c_2)^2 c_6^2}{16} +
  \frac{(c_1+c_2)^2 c_4^2 c_6^2 \chi^2}{4} + c_4 \ell +1$ and
  $c_{12}\equiv c_{12}(u,v,\ou,\ovv,T_0):=\frac{c_4^2}{4}$.\abs
  To finally control the two last integrals in (\ref{53.14}) and the rightmost summand in (\ref{53.16}), 
  we multiply the second equation in (\ref{53.1}) by $-\psi_j$ and hence see upon an integration by parts that
  \bea{53.17}
	& & \hs{-30mm}
	- \frac{1}{2} \frac{d}{dt} \io \psi_j^2
	+ \io |\na\psi_j|^2 \nn\\
	&=& - \chi \io \ou_j^2 \ovv_j \psi_j\Del\vp_j
	- \frac{1}{2} \io \psi_j \na\ou_j^2 \cdot\na\vp_j \nn\\
	& & + \chi \io \ou_j^2 \psi_j \na v_j\cdot\na\vp_j
	- \chi \io \ovv_j \psi_j \na\ou_j^2 \cdot\na\vp_j
	- \chi \io \ou_j^2 \psi_j \na\ovv_j\cdot\na\vp_j \nn\\
	& & + \ell \io \ou_j \vp_j \psi_j
	- \io \ou_j \psi_j^2
	\qquad \mbox{for all } t\in (0,T_0).
  \eea
  Several further applications of Young's inequality and (\ref{53.7}) reveal that here
  \bea{53.18}
	- \chi \io \ou_j^2 \ovv_j \psi_j\Del\vp_j
	&\le& \frac{1}{32} \io \ou_j v_j |\Del\vp_j|^2
	+ 8\chi^2 \io \frac{\ou_j^3 \ovv_j^2}{v_j} \psi_j^2 \nn\\
	&\le& \frac{1}{32} \io \ou_j v_j |\Del\vp_j|^2
	+ \frac{8c_3^3 c_5^2 \chi^2}{c_3} \io \psi_j^2
  \eea
  and
  \bea{53.19}
	\chi \io \ou_j^2 \psi_j \na v_j\cdot\na\vp_j
	\le \io |\na\vp_j|^2
	+ \frac{c_2^4 c_6^2 \chi^2}{4} \io \psi_j^2
  \eea
  and
  \bea{53.20}
	- \chi \io \ou_j^2 \psi_j \na\ovv_j\cdot\na\vp_j
	\le \io |\na\vp_j|^2
	+ \frac{c_2^4 c_7^2 \chi^2}{4} \io \psi_j^2
  \eea
  as well as
  \bea{53.21}
	\ell \io \ou_j \vp_j \psi_j
	\le \io \vp_j^2 
	+ \frac{c_2^2 \ell^2}{4} \io \psi_j^2
  \eea
  for all $t\in (0,T_0)$, while clearly
  \be{53.22}
	- \io \ou_j \psi_j^2 \le 0
  \ee
  for all $t\in (0,T_0)$.
  In the crucial second and fourth summand on the right of (\ref{53.17}), we once again integrate by parts,
  and only after that employ Young's inequality and (\ref{53.7}), to confirm that
  \bea{53.23}
	- \frac{1}{2} \io \psi_j \na \ou_j^2\cdot\na\vp_j
	&=& \frac{1}{2} \io \ou_j^2 \psi_j \Del\vp_j
	+ \frac{1}{2} \io \ou_j^2 \na\vp_j\cdot\na\psi_j \nn\\
	&\le& \frac{1}{64} \io \ou_j v_j |\Del\vp_j|^2
	+ \frac{4c_2^3}{c_3} \io \psi_j^2 \nn\\
	& & + \frac{1}{4} \io |\na\psi_j|^2
	+ \frac{c_2^4}{4} \io |\na\vp_j|^2
  \eea
  and
  \bea{53.24}
	- \chi \io \ovv_j \psi_j \na\ou_j^2\cdot\na\vp_j
	&=& \chi \io \ou_j^2 \ovv_j \psi_j \Del\vp_j
	+ \chi \io \ou_j^2 \ovv_j \na\vp_j\cdot\na\psi_j
	+ \chi \io \ou_j^2 \psi_j \na\ovv_j\cdot\na\vp_j \nn\\
	&\le& \frac{1}{64} \io \ou_j v_j |\Del\vp_j|^2
	+ \frac{16 c_2^3 c_5^2}{c_3} \io \psi_j^2 \nn\\
	& & + \frac{1}{4} \io |\na\psi_j|^2
	+ c_2^4 c_5^2 \chi^2 \io |\na\vp_j|^2 \nn\\
	& & + \io |\na\vp_j|^2
	+ \frac{c_2^4 c_7^2 \chi^2}{4} \io \psi_j^2
  \eea
  for all $t\in (0,T_0)$.
  In conjunction with (\ref{53.18})-(\ref{53.22}), the inequalities in (\ref{53.23}) and (\ref{53.24}) ensure that if we write
  $c_{13}\equiv c_{13}(u,v,\ou,\ovv,T_0):=3+\frac{c_2^4}{4}+c_2^4 c_5^2 \chi^2$ and
  $c_{14}\equiv c_{14}(u,v,\ou,\ovv,T_0):=\frac{8c_3^3 c_5^2 \chi^2}{c_3} + \frac{c_2^4 c_6^2 \chi^2}{4} + \frac{c_2^4 c_7^2 \chi^2}{4}
  + \frac{c_2^2 \ell^2}{4} + \frac{4c_2^3}{c_3} + \frac{16 c_2^3 c_5^2}{c_3} + \frac{c_2^4 c_7^2 \chi^2}{4}$, 
  then since $\ou_j\le u_j+\ou_j$,
  \bas
	& & \hs{-20mm}
	- \frac{1}{2} \frac{d}{dt} \io \psi_j^2
	+ \frac{1}{2} \io |\na\psi_j|^2 \nn\\
	&\le& \frac{1}{16} \io (u_j+\ou_j) v_j |\Del\vp_j|^2
	+ c_{13} \io |\na\vp_j|^2
	+ \io \vp_j^2
	+ c_{14} \io \psi_j^2
  \eas
  for all $t\in (0,T_0)$.
  On adding this to (\ref{53.14}) and (\ref{53.16}), we altogether obtain that
  \bea{53.25}
	& & \hs{-20mm}
	- \frac{d}{dt} \bigg\{ \io |\na\vp_j|^2 + \io \vp_j^2 + \io \psi_j^2 \bigg\} 
	+ \frac{1}{2j} \io v_j |\Del\vp_j|^2 \nn\\
	&\le& c_{15} \cdot \bigg\{ \io |\na\vp_j|^2 + \io \vp_j^2 + \io \psi_j^2 \bigg\}
	\qquad \mbox{for all } t\in (0,T_0)
  \eea
  with $c_{15}\equiv c_{15}(u,v,\ou,\ovv,T_0):=2\max\big\{ c_8+2+c_{13}, c_9+c_{11}+1 , c_{10}+c_{12}+c_{14}\}$.
  By means of a Gr\"onwall type inequality, we thus infer that
  for all $t\in [0,T_0]$ and $j\in\N$,
  \bea{53.26}
	\io |\na\vp_j|^2 + \io \vp_j^2 + \io \psi_j^2
	&\le& c_{16} \equiv c_{16}(u,v,\ou,\ovv,T_0,\vp_0,\psi_0) \nn\\
	&:=& e^{c_{15} T_0} \cdot \bigg\{ \io |\na\vp_0|^2 + \io \vp_0^2 + \io \psi_0^2 \bigg\},
  \eea
  whereupon an integration in (\ref{53.25}) shows that
  \be{53.27}
	\frac{1}{2j} \int_0^{T_0} \io v_j |\Del\vp_j|^2
	\le c_{16} + c_{15} c_{16} T_0
	\qquad \mbox{for all } j\in\N.
  \ee
  With some suitably large $\Gamma=\Gamma(u,v,\ou,\ovv,T_0,\vp_0,\psi_0)>0$, we hence conclude (\ref{53.3}) and (\ref{53.4})
  from (\ref{53.26}), while (\ref{53.6}) results from (\ref{53.27}) and the fact that $v_j\ge c_3$ in $\Om\times (0,T_0)$
  by (\ref{53.7}).
\qed
Using the solutions to (\ref{53.1}) constructed above as test functions in (\ref{51.2}), now drawing on 
the gradient regularity property in (\ref{52.8}) in a corresponding limit passage we can accomplish the main step toward
our proof of Theorem \ref{theo55}:
\begin{lem}\label{lem54}
  Let $n\ge 1, \chi\ge 0$ and $\ell\ge 0$, assume (\ref{i1}), let $(u,v)$ 
  and $(\ou,\ovv)$ be two locally bounded weak solutions of (\ref{0}) in $\Om\times (0,T)$ for some $T\in (0,\infty]$,
  let $\vp_0\in C_0^\infty(\Om)$ and $\psi_0\in C_0^\infty(\Om)$,
  and let $N=N(u,v,\ou,\ovv,\vp_0,\psi_0)$ be as in Lemma \ref{lem51}.
  Then for each $T_0\in (0,T)\sm N$,
  \be{54.1}
	\io \big( u(\cdot,T_0)-\ou(\cdot,T_0)\big) \vp_0
	+ \io \big( v(\cdot,T_0)-\ovv(\cdot,T_0)\big) \psi_0
	= 0.
  \ee
\end{lem}
\proof
  Given $T_0\in (0,T)\sm N$, we let $(\vp_j)_{j\in\N}$ and $(\psi_j)_{j\in\N}$ be as in Lemma \ref{lem53}, and 
  taking $w$ and $z$ from (\ref{wz}) and combining 
  (\ref{53.1}) with (\ref{51.2}) we obtain the identity
  \bea{54.2}
	& & \hs{-20mm}
	\io w(\cdot,T_0) \vp_0 
	+ \io z(\cdot,T_0) \psi_0 \nn\\
	&=& - \frac{1}{2} \int_0^{T_0} \io w\cdot\Big\{ \Big(u_j+\ou_j+\frac{1}{j}\Big) v_j
	- (u+\ou) v\Big\} \Del\vp_j \nn\\
	& & - \frac{1}{2} \int_0^{T_0} \io w\cdot\Big\{ (u_j+\ou_j)\na v_j - (u+\ou)\na v\Big\} \cdot\na\vp_j \nn\\
	& & - \chi \int_0^{T_0} \io  w\cdot\Big\{ (u_j+\ou_j)v_j \na v_j - (u+\ou) v\na v\Big\} \cdot\na\vp_j \nn\\
	& & - \ell \int_0^{T_0} \io w (v_j-v) \vp_j
	+ \int_0^{T_0} \io w (v_j-v) \psi_j \nn\\
	& & + \chi \int_0^{T_0} \io z (\ou_j^2 \ovv_j - \ou^2 \ovv) \Del\vp_j 
	+\frac{1}{2} \int_0^{T_0} \io z (\na\ou_j^2 - \na \ou^2)\cdot\na\vp_j \nn\\
	& & - \chi \int_0^{T_0} \io z (\ou_j^2 \na v_j - \ou^2 \na v)\cdot\na\vp_j 
	+ \chi \int_0^{T_0} \io z(\ovv_j \na\ou_j^2 - \ovv\na \ou^2) \cdot\na\vp_j \nn\\
	& & + \chi \int_0^{T_0} \io z (\ou_j^2 \na\ovv_j - \ou^2 \na\ovv) \cdot\na\vp_j \nn\\
	& & - \ell \int_0^{T_0} \io z(\ou_j-\ou)\vp_j
	+ \int_0^{T_0} \io z (\ou_j-\ou)\psi_j
	\qquad \mbox{for all } j\in\N.
  \eea
  In order to estimate each of the integrals on the right-hand side herein, we use the boundedness of $u,\ou,\na v$ and $\na\ovv$
  in $\Om\times (0,T_0)$ when fixing $c_i>0$, $i\in\{1,...,8\}$, such that
  \be{54.3}
	|w|\le c_1,
	\quad u\le c_2,
	\quad \ou\le c_3,
	\quad |z|\le c_4,
	\quad v\le c_5,
	\quad \ovv\le c_6,
	\quad |\na v|\le c_7
	\quad \mbox{and} \quad
	|\na\ovv| \le c_8
  \ee
  in $\Om\times (0,T_0)$,
  and recall Lemma \ref{lem52} in choosing positive constants $c_i$, $i\in\{9,...,15\}$, in such a way that
  \be{54.4}
	u_j \le c_9,
	\quad \ou_j \le c_{10},
	\quad v_j \le c_{11},
	\quad \ovv_j \le c_{12},
	\quad |\na v_j|\le c_{13}
	\quad \mbox{and} \quad
	|\na\ovv_j|\le c_{14}
  \ee
  in $\Om\times (0,T_0)$ as well as
  \be{54.44}
	\int_0^{T_0} \io |\na \ou_j^2|^2 \le c_{15}
  \ee
  for all $j\in\N$.
  Moreover taking $\Gamma=\Gamma(u,v,\ou,\ovv,T_0,\vp_0,\psi_0)$ from Lemma \ref{lem53}, by means of the Cauchy-Schwarz inequality,
  (\ref{54.3}), (\ref{54.4}) and (\ref{53.6}) we can thus estimate
  \bas
	& & \hs{-20mm}
	\bigg| -\frac{1}{2} \int_0^{T_0} \io w\cdot\Big\{ \Big( u_j+\ou_j+\frac{1}{j}\Big) v_j - (u+\ou)v\Big\} \Del\vp_j \bigg|
		\nn\\
	&\le& \frac{1}{2} \cdot\bigg\{ \int_0^{T_0} \io |\Del\vp_j|^2\bigg\}^\frac{1}{2} \cdot
		\bigg\{ \int_0^{T_0} \io w^2 \cdot \Big\{ \Big(u_j+\ou_j+\frac{1}{j}\Big) v_j - (u+\ou)v\Big\}^2 \bigg\}^\frac{1}{2}
		\nn\\
	&\le& \frac{c_1 (\Gamma j)^\frac{1}{2}}{2} \Big\| \Big(u_j+\ou_j+\frac{1}{j}\Big) v_j - (u+\ou)v\Big\|_{L^2(\Om\times (0,T_0))}
		\nn\\
	&=& \frac{c_1 (\Gamma j)^\frac{1}{2}}{2} \Big\| (u_j-u) v_j + (\ou_j-\ou)v_j + (u+\ou)(v_j-v) + \frac{1}{j} v_j
		\Big\|_{L^2(\Om\times (0,T_0))} \nn\\
	&\le& \frac{c_1 (\Gamma j)^\frac{1}{2}}{2} \cdot \Big\{
	c_{11} \|u_j-u\|_{L^2(\Om\times (0,T_0))} + c_{11} \|\ou_j-\ou\|_{L^2(\Om\times (0,T_0))} \nn\\
	& & \hs{10mm}
	+ (c_2+c_3) \|v_j-v\|_{L^2(\Om\times (0,T_0))} + \frac{c_{11}}{j}\Big\} \nn\\
	&=& \frac{c_1 c_{11} \Gamma^\frac{1}{2}}{2} \cdot \bigg\{ j \int_0^{T_0} \io (u_j-u)^2\bigg\}^\frac{1}{2}
	+ \frac{c_1 c_{11} \Gamma^\frac{1}{2}}{2} \cdot \bigg\{ j \int_0^{T_0} \io (\ou_j-\ou)^2\bigg\}^\frac{1}{2} \nn\\
	& & + \frac{c_1 (c_2+c_3) \Gamma^\frac{1}{2}}{2} \cdot \bigg\{ j \int_0^{T_0} \io (v_j-v)^2 \bigg\}^\frac{1}{2}
	+ \frac{c_1 c_{11} \Gamma^\frac{1}{2}}{2j^\frac{1}{2}} 
	\qquad \mbox{for all } j\in\N,
  \eas
  so that thanks to (\ref{52.6}) and (\ref{52.7}),
  \be{54.5}
	-\frac{1}{2} \int_0^{T_0} \io w\cdot\Big\{ \Big( u_j+\ou_j+\frac{1}{j}\Big) v_j - (u+\ou)v\Big\} \Del\vp_j 
	\to 0
	\qquad \mbox{as } j\to\infty.
  \ee
foo
  Next, the Cauchy-Schwarz inequality, (\ref{53.3}), (\ref{54.3}) and (\ref{54.4}) imply that
  \bas
	& & \hs{-20mm}
	\bigg| - \frac{1}{2} \int_0^{T_0} \io w\cdot\Big\{ (u+\ou_j)\na v_j - (u+\ou)\na v\Big\}\cdot\na\vp_j\bigg| \nn\\
	&\le& \frac{c_1}{2} \cdot \bigg\{ \int_0^{T_0} \io |\na\vp_j|^2 \bigg\}^\frac{1}{2} \cdot
		\Big\| (u_j+\ou_j)\na v_j - (u+\ou)\na v\Big\|_{L^2(\Om\times (0,T_0))} \\
	&\le& \frac{c_1 (\Gamma T_0)^\frac{1}{2}}{2} \Big\| (u_j+\ou_j)\na v_j - (u+\ou)\na v\Big\|_{L^2(\Om\times (0,T_0))} \\
	&=& \frac{c_1 (\Gamma T_0)^\frac{1}{2}}{2} \Big\| (u_j-u)\na v_j + (\ou_j-\ou)\na v_j + (u+\ou)(\na v_j-\na v)
		\Big\|_{L^2(\Om\times (0,T_0))} \nn\\
	&\le& \frac{c_1 c_{13} (\Gamma T_0)^\frac{1}{2}}{2} \|u_j-u\|_{L^2(\Om\times (0,T_0))}
	+ \frac{c_1 c_{13} (\Gamma T_0)^\frac{1}{2}}{2} \|\ou_j-\ou\|_{L^2(\Om\times (0,T_0))} \nn\\
	& & + \frac{c_1 (c_2+c_3) (\Gamma T_0)^\frac{1}{2}}{2} \|\na v_j - \na v\|_{L^2(\Om\times (0,T_0))}
	\qquad \mbox{for all } j\in\N,
  \eas
  so that using (\ref{52.9}) and again (\ref{52.6}) we find that
  \be{54.6}
	- \frac{1}{2} \int_0^{T_0} \io w\cdot\Big\{ (u+\ou_j)\na v_j - (u+\ou)\na v\Big\}\cdot\na\vp_j
	\to 0
	\qquad \mbox{as } j\to\infty.
  \ee
  Similarly, in addition utilizing (\ref{52.7}) and (\ref{53.4}), we see that
  \bea{54.7}
	& & \hs{-16mm}
	\bigg| - \chi \int_0^{T_0} \io w\cdot\Big\{ (u_j+\ou_j)v_j \na v_j - (u+\ou) v\na v\Big\} \cdot\na\vp_j \bigg| \nn\\
	&\le& c_1 \chi (\Gamma T_0)^\frac{1}{2} 
		\Big\| (u_j+\ou_j)v_j \na v_j - (u+\ou)v\na v\Big\|_{L^2(\Om\times (0,T_0))} \nn\\
	&=& c_1\chi (\Gamma T_0)^\frac{1}{2} 
		\Big\| (u_j-u)v_j \na v_j + (\ou_j-\ou) v_j\na v_j \nn\\
	& & \hs{20mm}
	+ (u+\ou)(v_j-v)\na v_j + (u+\ou)v(\na v_j-\na v)
		\Big\|_{L^2(\Om\times (0,T_0))} \nn\\
	&\le& c_1 c_{11} c_{13} (\Gamma T_0)^\frac{1}{2} \|u_j-u\|_{L^2(\Om\times (0,T_0))}
	+ c_1 c_{11} c_{13} (\Gamma T_0)^\frac{1}{2} \|\ou_j-\ou\|_{L^2(\Om\times (0,T_0))} \nn\\
	& & + c_1 (c_2+c_3) c_{13} (\Gamma T_0)^\frac{1}{2} \|v_j-v\|_{L^2(\Om\times (0,T_0))}
	+ c_1(c_2+c_3) c_5 (\Gamma T_0)^\frac{1}{2} \|\na v_j-\na v\|_{L^2(\Om\times (0,T_0))} \nn\\[2mm]
	&\to& 0
	\qquad \mbox{as } j\to\infty
  \eea
  and
  \bea{54.8}
	\bigg| - \ell \int_0^{T_0} \io w(v_j-v) \vp_j\bigg|
	\le c_1 \ell (\Gamma T_0)^\frac{1}{2} \|v_j-v\|_{L^2(\Om\times (0,T_0))}
	\to 0
  \eea
  as well as
  \be{54.9}
	\bigg| \int_0^{T_0} \io w(v_j-v)\psi_j\bigg|
	\le c_1(\Gamma T_0)^\frac{1}{2} \|v_j-v\|_{L^2(\Om\times (0,T_0))}
	\to 0
  \ee
  as $j\to\infty$.
  Apart from that, again by (\ref{53.6}), (\ref{52.6}) and (\ref{52.7}),
  \bea{54.10}
	& & \hs{-20mm}
	\bigg| \chi \int_0^{T_0} \io z(\ou_j^2 \ovv_j - \ou^2 \ovv)\Del\vp_j\bigg| \nn\\
	&\le& c_4 \chi \cdot\bigg\{ \int_0^{T_0} \io |\Del\vp_j|^2 \bigg\}^\frac{1}{2} \cdot 
		\| \ou_j^2 \ovv_j - \ou^2 \ovv\|_{L^2(\Om\times (0,T_0))} \nn\\
	&\le& c_4 \chi (\Gamma j)^\frac{1}{2} 	
		\Big\| (\ou_j+\ou) (\ou_j-\ou)\ovv_j + \ou^2 (\ovv_j-\ovv)\Big\|_{L^2(\Om\times (0,T_0))} \nn\\
	&\le& (c_3+c_{10}) c_4 c_{12} \Gamma^\frac{1}{2} \cdot \bigg\{ j \int_0^{T_0} \io (\ou_j-\ou)^2 \bigg\}^\frac{1}{2} 
	+ c_3^2 c_4 \chi \Gamma^\frac{1}{2} \cdot \bigg\{ \int_0^{T_0} \io (\ovv_j-\ovv)^2 \bigg\}^\frac{1}{2} \nn\\[2mm]
	&\to& 0
	\qquad \mbox{as } j\to\infty,
  \eea
  while by (\ref{52.8}),
  \bea{54.11}
	\bigg| \frac{1}{2} \int_0^{T_0} \io z (\na\ou_j^2 -\na\ou^2)\cdot\na\vp_j\bigg|
	&\le& \frac{c_4 (\Gamma T_0)^\frac{1}{2}}{2} \cdot \bigg\{ \int_0^{T_0} \io |\na \ou_j^2-\na\ou^2|^2\bigg\}^\frac{1}{2} 
		\nn\\[2mm]
	&\to& 0
	\qquad \mbox{as } j\to\infty
  \eea
  and by (\ref{52.6}) and (\ref{52.9}),
  \bea{54.12}
	& & \hs{-20mm}
	\bigg| - \chi \int_0^{T_0} \io z(\ou_j^2 \na v_j-\ou^2 \na v)\cdot\na\vp_j\bigg| \nn\\
	&\le& c_4 (\Gamma T_0)^\frac{1}{2} \chi \|\ou_j^2\na v_j - \ou^2 \na v\|_{L^2(\Om\times (0,T_0))} \nn\\
	&=& c_4 (\Gamma T_0)^\frac{1}{2} \chi \Big\| (\ou_j+\ou)(\ou_j-\ou)\na v_j + \ou^2 (\na v_j-\na v)
		\Big\|_{L^2(\Om\times (0,T_0))} \nn\\
	&\le& (c_3+c_{10}) c_4 c_{13} (\Gamma T_0)^\frac{1}{2} \chi \|\ou_j-\ou\|_{L^2(\Om\times (0,T_0))}
	+ c_3^2 c_4 (\Gamma T_0)^\frac{1}{2} \chi \|\na v_j-\na v\|_{L^2(\Om\times (0,T_0))} \nn\\[2mm]
	&\to& 0
	\qquad \mbox{as } j\to\infty.
  \eea
  Furthermore, since $\na \ou_j^2 \to \na \ou^2$ in $L^2(\Om\times (0,T_0))$ and hence 
  $|\na \ou_j^2|^2 \to |\na \ou^2|^2$ in $L^1(\Om\times (0,T_0))$ as $j\to\infty$ according to (\ref{52.8}), and since
  $(v_j-v)^2\wsto 0$ in $L^\infty(\Om\times (0,T_0))$ as $j\to\infty$ thanks to (\ref{52.5}) and (\ref{54.4}),
  it follows that
  \bas
	\int_0^{T_0} \io (\ovv_j-\ovv)^2 |\na\ou_j^2|^2 \to 0
	\qquad \mbox{as } j\to\infty
  \eas
  and thus
  \bea{54.13}
	& & \hs{-20mm}
	\bigg| \chi \int_0^{T_0} \io z (\ovv_j \na\ou_j^2 - \ovv \na \ou^2)\cdot\na\vp_j\bigg| \nn\\
	&\le& c_4 \chi (\Gamma T_0)^\frac{1}{2} \|\ovv_j \na\ou_j^2 - \ovv \na \ou^2\|_{L^2(\Om\times (0,T_0))} \nn\\
	&=& c_4 \chi (\Gamma T_0)^\frac{1}{2} 
		\Big\| (\ovv_j-\ovv)\na\ou_j^2 + \ovv(\na\ou_j^2-\na\ou^2)\Big\|_{L^2(\Om\times (0,T_0))} \nn\\
	&\le& c_4 \chi (\Gamma T_0)^\frac{1}{2} \cdot \bigg\{ \int_0^{T_0} \io (\ovv_j-\ovv)^2 |\na\ou_j^2|^2 \bigg\}^\frac{1}{2}
	+ c_4 c_6 \chi (\Gamma T_0)^\frac{1}{2} \|\na\ou_j^2 - \na \ou^2\|_{L^2(\Om\times (0,T_0))} \nn\\[2mm]
	&\to& 0
	\qquad \mbox{as } j\to\infty,
  \eea
  whereas similarly using that
  \bas
	\int_0^{T_0} \io (\ou_j-\ou)^2 |\na \ovv_j|^2 \to 0
	\qquad \mbox{as } j\to\infty
  \eas
  due to (\ref{52.5}), (\ref{52.9}) and (\ref{54.4}), we obtain that also
  \bea{54.14}
	& & \hs{-14mm}
	\bigg| \chi \int_0^{T_0} \io z (\ou_j^2\na\ovv_j - \ou^2 \na\ovv)\cdot\na\vp_j\bigg| \nn\\
	&\le& c_4 \chi (\Gamma T_0)^\frac{1}{2} \|\ou_j^2 \na\ovv_j - \ou^2 \na\ovv\|_{L^2(\Om\times (0,T_0))} \nn\\
	&=& c_4 \chi (\Gamma T_0)^\frac{1}{2} 
		\Big\| (\ou_j+\ou)(\ou_j-\ou) \na\ovv_j + \ou^2 (\na\ovv_j-\na\ovv)\Big\|_{L^2(\Om\times (0,T_0))} \nn\\
	&\le& (c_3+c_{10}) c_4 \chi (\Gamma T_0)^\frac{1}{2} \cdot \bigg\{ \int_0^{T_0} \io (\ou_j-\ou)^2 |\na\ovv_j|^2 
		\bigg\}^\frac{1}{2}
	+ c_3^2 c_4 \chi (\Gamma T_0)^\frac{1}{2} \|\na\ovv_j-\na\ovv\|_{L^2(\Om\times (0,T_0))} \nn\\[2mm]
	&\to& 0
	\qquad \mbox{as } j\to\infty.
  \eea
  Since, finally, 
  \be{54.15}
	\bigg| -\ell \int_0^{T_0} \io z(\ou_j-\ou)\vp_j\bigg|
	\le c_4 \ell (\Gamma T_0)^\frac{1}{2} \|\ou_j-\ou\|_{L^2(\Om\times (0,T_0))}
	\to 0
	\qquad \mbox{as } j\to\infty
  \ee
  and
  \be{54.16}
	\bigg| \int_0^{T_0} \io z(\ou_j-\ou)\psi_j\bigg|
	\le c_4 (\Gamma T_0)^\frac{1}{2} \|\ou_j-\ou\|_{L^2(\Om\times (0,T_0))}
	\to 0
	\qquad \mbox{as } j\to\infty
  \ee
  according to (\ref{54.3}), (\ref{53.3}), (\ref{53.4}) and (\ref{52.6}), upon collecting (\ref{54.5})-(\ref{54.16})
  we infer (\ref{54.1}) as a consequence of (\ref{54.2}).
\qed
In fact, thanks to the flexibility left above with respect to the choice of $\vp$ and $\psi$ in (\ref{54.1}),
by means of a final and fairly simple approximation argument we can complete our derivation of the claimed uniqueness property:\abs
\proofc of Theorem \ref{theo55}. \quad
  Let $(\vp_{0j})_{j\in\N} \subset C_0^\infty(\Om)$ and $(\psi_{0j})_{j\in\N} \subset C_0^\infty(\Om)$ be such that
  $\{\vp_{0j} \ | \ j\in\N\}$ and $\{\psi_{0j} \ | \ j\in\N\}$ are both dense subsets of $L^1(\Om)$ containing the trivial function,
  and for $j\in\N$ letting $N(u,v,\ou,\ovv,\vp_{0j},\psi_{0j})$ denote the corresponding null sets obtained in Lemma \ref{lem51},
  we define $N_\star:=\bigcup_{j\in\N} N(u,v,\ou,\ovv,\vp_{0j},\psi_{0j})$.
  For arbitrary $T_0\in (0,T)\sm N_\star$, we may then apply (\ref{54.1}) to $\vp_0:=\vp_{0j}$ and $\psi:=0$ to see that
  \bas
	\io \big( u(\cdot,T_0)-\ou(\cdot,T_0)\big) \vp_{0j}=0
	\qquad \mbox{for all } j\in\N,
  \eas
  which by the density property of $\{\vp_{0j} \ | \ j\in\N\}$ implies that $u(\cdot,T_0)=\ou(\cdot,T_0)$ a.e.~in $\Om$.
  Since it can similarly be seen that $v(\cdot,T_0)=\ovv(\cdot,T_0)$ a.e.~in $\Om$, from the fact that $|N_\star|=0$ and the 
  Fubini-Tonelli theorem we infer that, indeed, $(u,v)=(\ou,\ovv)$ a.e.~in $\Om\times (0,T)$.
\qed
\mysection{Continuous dependence. Proof of Theorems \ref{theo33} and \ref{theo34}}
In order to prepare our analysis related to Theorem \ref{theo33} and Theorem \ref{theo34}, let us import 
some results known from the literature on (\ref{0}) in settings of initial data which are Lipschitz
continuous in both components.
The following extracts and summarizes from the one-, two- and three-dimensional studies in 
\cite{liwin_CPAA}, \cite{win_wu}, \cite{win_JDE}, \cite{win_wu2} and \cite{taowin_DCDS2027} what will be needed in the sequel;
we remark here that in line with their respective focal points these precedents essentially do not formulate
the dependencies of estimates such as that in (\ref{333.5}) on the parameters $K$ and $\LL$ appearing below explicitly,
but that these can readily be verified by trivial adaptation in the corresponding arguments.
\begin{lem}\label{lem333}
  Let $n\le 3$.
  Then there exist $(\del_1(K,\LL))_{K>0,\LL>0} \subset (0,\infty)$ and $(\Lam_1(K,\LL))_{K>0,\LL>0} \subset (0,\infty)$
  with the property that
  \be{333.1}
	\lball
	\del_1(K,\LL) = +\infty
	\quad \mbox{for all $K>0$ and $\LL>0$}
	\qquad & \mbox{if $n\le 2$ and $\Om$ is convex,} \\[1mm]
	\del_1(K,\LL) \to + \infty
	\quad \mbox{for all $K>0$ as $\LL\searrow 0$}
	\qquad & \mbox{if $n\le 3$ and $\Om$ is arbitrary,}
	\ear
  \ee
  and that if
  $\chi\in [0,\LL)$ and $\ell\in [0,\LL)$, if
  \be{init}
	\lbal
	u_0\in W^{1,\infty}(\Om)
	\mbox{ is nonnegative,		
		\quad and that} \\[1mm]
	v_0 \in W^{1,\infty}(\Om)	
	\mbox{ is such that $v_0>0$ in $\bom$,}
	\ear
  \ee
  and if
  \be{333.2}
	\|u_0\|_{L^\infty(\Om)} < K,
	\quad
	\io u_0>\frac{1}{K},
	\quad
	\|v_0\|_{L^\infty(\Om)} < K
	\quad \mbox{and} \quad
	\Big\|\frac{\na v_0}{v_0}\Big\|_{L^\infty(\Om)} < K
  \ee
  as well as
  \be{333.3}
	\io v_0 < \del_1(K,\LL),
  \ee
  then for each $\eps\in (0,1)$ there exists a global classical solution $(\ueps,\veps)$ of (\ref{0}) with initial data 
  $(u_0+\eps,v_0)$, satisfying
  \be{333.4}
	\lbal
	\ueps\in C^0(\bom\times [0,\infty)) \cap C^{2,1}(\bom\times (0,\infty))
	\qquad \mbox{and} \\[1mm]
	\veps\in\bigcap_{q>n} C^0([0,\infty);W^{1,q}(\Om)) \cap C^{2,1}(\bom\times (0,\infty))
	\ear
  \ee
  as well as $\ueps>0$ and $\veps>0$ in $\bom\times [0,\infty)$, and such that
  \be{333.5}
	\|\ueps(\cdot,t)\|_{L^\infty(\Om)} \le \Lambda_1(K,\LL)
	\qquad \mbox{for all } t>0.
  \ee
\end{lem}
Based on $L^\infty$ estimates of the form in (\ref{333.5}), the corresponding second solution components can be seen 
to satisfy certain elliptic Harnack inequalities. More precisely, Lemma \ref{lem41} below will crucially rely on the following
consequence of a pointwise property in parabolic problems, drawn in \cite[Lemma 6.2]{win_wu2} on the basis of a result
from \cite{huska}.
\begin{lem}\label{lem_harnack}
  Let $n\ge 1$ and $K_0>0$.
  Then there exists $\lam(K_0)>0$ with the property that whenever $T>0$,
  $f\in C^0(\bom\times [0,T))$, $w_0\in W^{1,\infty}(\Om)$ and
  $w\in \bigcap_{q>n} C^0([0,T);W^{1,q}(\Om)) \cap C^{2,1}(\bom\times (0,T))$ are such that
  $w_0>0$ in $\bom$ and
  \be{22.1}
	\Big\| \frac{\na w_0}{w_0}\Big\|_{L^\infty(\Om)} \le K_0
  \ee
  and
  \be{22.2}
	\|f(\cdot,t)\|_{L^\infty(\Om)} \le K_0
	\qquad \mbox{for all } t\in (0,T),
  \ee
  and that
  \be{0w}
	\lball
	w_t = \Del w + f(x,t) w,
	\qquad & x\in \Om, \, t\in (0,T), \\[1mm]
	\frac{\pa w}{\pa\nu}=0,
	\qquad & x\in \pO, \, t\in (0,T), \\[1mm]
	w(x,0)=w_0(x),
	\qquad & x\in \Om,
	\ear
  \ee
  it follows that
  \be{22.3}
	w(x,t) \ge \lam(K_0) \|w(\cdot,t)\|_{L^\infty(\Om)}
	\qquad \mbox{for all $x\in \Om$ and } t\in (0,T).
  \ee
\end{lem}
Now the core of our reasoning in this section can be found in the following argument which facilitates a simultaneous
treatment of the existence statement made in Theorem \ref{theo55}, and of the claim concerning continuous dependence appearing
therein. In a key place, the following lemma relies on the uniqueness property asserted by Theorem \ref{theo55}.
\begin{lem}\label{lem41}
  Let $n\ge 1$, and suppose that 
  $(\chi_k)_{k\in\N} \subset [0,\infty)$, $(\ell_k)_{k\in\N} \subset [0,\infty)$,
  $(u_{0k})_{k\in\N} \subset W^{1,\infty}(\Om)$ and $(v_{0k})_{k\in\N} \subset W^{1,\infty}(\Om)$
  are such that $u_{0k}\ge 0$ and $v_{0k}>0$ in $\bom$ for all $k\in\N$, that
  as $k\to\infty$ we have
  \be{41.01}
	\chi_k\to \chi
	\qquad \mbox{and} \qquad
	\ell_k\to\ell
  \ee
  as well as
  \be{41.1}
	u_{0k} \to u_0
	\quad \mbox{in } L^\infty(\Om)
	\qquad \mbox{and} \qquad
	v_{0k} \to v_0
	\quad \mbox{in } W^{1,\infty}(\Om)
  \ee
  with some $\chi\ge 0$, $\ell\ge 0$, $u_0$ and $v_0$ fulfilling (\ref{i2}), 
  and that for each $k\in\N$ the problem (\ref{0}) with parameters $\chi_k$ and $\ell_k$ admits a global classical
  solution $(u_k,v_k)$ evolving from the initial data $(u_{0k},v_{0k})$, with
  \bas
	\lbal
	u_k\in C^0(\bom\times [0,\infty)) \cap C^{2,1}(\bom\times (0,\infty))
	\qquad \mbox{and} \\[1mm]
	v_k\in\bigcap_{q>n} C^0([0,\infty);W^{1,q}(\Om)) \cap C^{2,1}(\bom\times (0,\infty)).
	\ear
  \eas
  If the family of these solutions satisfies
  \be{41.2}
	\sup_{k\in\N} \|u_k\|_{L^\infty(\Om\times (0,\infty))} < \infty,
  \ee
  then there exists a unique locally bounded weak solution $(u,v)$ of (\ref{0}) in $\Om\times (0,\infty)$ 
  with initial data $(u_0,v_0)$, and we have
  \be{41.21}
	\lbal
	u\in C^0(\bom\times [0,\infty)) 
	\qquad \mbox{and} \\[1mm]
	v\in C^0(\bom\times [0,\infty)) \cap C^{2,1}(\bom\times (0,\infty))
	\ear
  \ee
  as well as
  \be{41.22}
	\|u\|_{L^\infty(\Om\times (0,\infty))}
	\le \sup_{k\in\N} \|u_k\|_{L^\infty(\Om\times (0,\infty))} 
  \ee
  and 
  \be{41.3}
	\|u_k-u\|_{L^\infty(\Om\times (0,\infty))}
	+ \|v_k-v\|_{L^\infty(\Om\times (0,\infty))}
	\to 0
	\qquad \mbox{as } k\to\infty.
  \ee
\end{lem}
\proof
  Using that (\ref{i2}), (\ref{41.1}) and (\ref{41.2}) ensure positivity of $c_1:=\inf_{k\in\N} \inf_{x\in\Om} v_{0k}(x)$
  and finiteness of $c_2:=\sup_{k\in \N} \|u_k\|_{L^\infty(\Om\times (0,\infty))}$, we may first employ a comparison principle to see 
  that since $\un{v}(x,t):=c_1 e^{-c_2 t}$, $(x,t)\in\bom\times [0,\infty)$, satisfies
  $\un{v}_t-\Del \un{v}+ u_k\un{v} = c_1\cdot (-c_2+u_k) \cdot e^{-c_2 t} \le 0$ in $\Om\times (0,\infty)$
  for all $k\in\N$, it follows that $v_k(x,t) \ge c_1 e^{-c_2 t}$ for all $x\in\Om, t>0$ and $k\in\N$, in particular meaning that
  \be{41.4}
	\inf_{k\in\N} \inf_{(x,t)\in\Om\times (0,T)} v_k(x,t) > 0
	\qquad \mbox{for all } T>0.
  \ee
  Since (\ref{41.2}) moreover entails that due to a standard argument based on heat semigroup estimates we have
  \be{41.43}
	\sup_{k\in\N} \sup_{t>0} \|v_k(\cdot,t)\|_{W^{1,\infty}(\Om)} < \infty,
  \ee
  and since (\ref{41.01}) and (\ref{41.1}) especially ensure boundedness of $(\chi_k)_{k\in\N}$ and $(\ell_k)_{k\in\N}$
  as well as equi-continuity of $(u_{0k})_{k\in\N}$, we may rely on a standard
  result on continuity in scalar parabolic equations with degeneracies of porous medium type 
  (\cite[Theorem 1.3 and Remark 1.4]{porzio_vespri}) to see that for each $T>0$, with some nondecreasing
  $\om_1: [0,\infty) \to [0,\infty)$ fulfilling $\om_1(\xi)\to 0$ as $\xi\to 0$,
  \bas
	|u_k(x_1,t_1)-u_k(x_2,t_2)|
	&\le& \om_1\big( |x_1-x_2| + |t_1-t_2|^\frac{1}{2}\big) \\
	& & \qquad \mbox{for all $(x_1,t_1)\in\Om\times (0,T), (x_2,t_2)\in\Om\times (0,T)$ and } k\in\N,
  \eas
  and that whenever $\tau>0$ and $T>\tau$, there exists $\theta_1=\theta_1(\tau,T)\in (0,1)$ such that
  \be{41.5}
	\sup_{k\in\N} \|u_k\|_{C^{\theta_1,\frac{\theta_1}{2}}(\bom\times [\tau,T])} < \infty.
  \ee
  Likewise, on the basis of (\ref{41.1}) and (\ref{41.2}) we can similarly conclude that given $T>0$ one can choose
  $\theta_2=\theta_2(T)\in (0,1)$ such that
  \be{41.6}
	\sup_{k\in\N} \|v_k\|_{C^{\theta_2,\frac{\theta_k}{2}}(\bom\times [0,T])} < \infty,
  \ee
  whereupon parabolic Schauder theory (\cite{LSU}) drawing on 
  the boundedness of $(\chi_k)_{k\in\N}$ and $(\ell_k)_{k\in\N}$ as well as on
  (\ref{41.5}) and (\ref{41.6}) implies that for each $\tau>0$ and any $T>\tau$
  there exists $\theta_3=\theta_3(\tau,T)\in (0,1)$ satisfying
  \be{41.7}
	\sup_{k\in\N} \|u_k\|_{C^{2+\theta_3,1+\frac{\theta_3}{2}}(\bom\times [\tau,T])} < \infty.
  \ee
  According to the compactness properties implied by (\ref{41.43})-(\ref{41.7}) through the Banach-Alaoglu and the
  Arzel\`a-Ascoli theorems, thanks to the uniqueness property asserted by Theorem \ref{theo55} the present lemma will thus
  be proved as soon as we have made sure that whenever $(k_l)_{l\in\N} \subset\N$,
  $u\in C^0(\bom\times [0,\infty))$ and 
  $v\in C^0(\bom\times [0,\infty)) \cap C^{2,1}(\bom\times (0,\infty)) \cap L^\infty((0,\infty);W^{1,\infty}(\Om))$
  are such that $k_l\nearrow\infty$ and
  \be{41.71}
	u_{k_l} \to u
	\qquad \mbox{in } C^0_{loc}(\bom\times [0,\infty))
  \ee
  as well as
  \be{41.72}
	v_{k_l} \to v
	\qquad \mbox{in } C^0_{loc}(\bom\times [0,\infty)) 
	\mbox{ and in } C^{2,1}_{loc}(\bom\times (0,\infty))
  \ee
  and
  \be{41.73}
	\na v_{k_l} \wsto \na v
	\qquad \mbox{in } L^\infty(\Om\times (0,\infty))
  \ee
  as $l\to\infty$, it follows that $(u,v)$ is a locally bounded weak solution of (\ref{0}) in $\Om\times (0,\infty)$
  with initial data $(u_0,v_0)$, which is such that
  \be{41.8}
	\|u_{k_l}-u\|_{L^\infty(\Om\times (0,\infty))}
	+ \|v_{k_l}-v\|_{L^\infty(\Om\times (0,\infty))}
	\to 0
	\qquad \mbox{as } l\to\infty.
  \ee
  To verify this, we once again rely on (\ref{41.2}) and (\ref{41.1}) in employing Lemma \ref{lem_harnack}, hence finding 
  $c_3>0$ such that
  \be{41.9}
	v_k(x,t)\ge c_3 \|v_k(\cdot,t)\|_{L^\infty(\Om)}
	\qquad \mbox{for all $x\in\Om, t>0$ and } k\in\N.
  \ee
  Since (\ref{41.1}) together with the hypothesis $u_0\not\equiv 0$ ensures the existence of $k_\star\in\N$, $c_4>0$ and $c_5>0$
  such that
  \be{41.10}
	\io u_{0k} \ge c_4
	\quad \mbox{and} \quad
	\io v_{0k} \le c_5
	\qquad \mbox{for all } k>k_\star,
  \ee
  and since thus
  \bas
	\io u_k(\cdot,t) \ge c_4
	\qquad \mbox{for all $t>0$ and } k>k_\star 
  \eas
  due to the fact that $\frac{d}{dt} \io u_k=\ell_k \io u_k v_k\ge 0$ for all $t>0$ and $k\in\N$ by the first equation in (\ref{0}),
  using (\ref{41.9}) along with the second equation in (\ref{0}) we obtain that
  \bas
	\frac{d}{dt} \io v_k
	&=& - \io u_k v_k
	\le c_3 \|v_k\|_{L^\infty(\Om)} \io u_k
	\le -c_3 c_4 \|v_k\|_{L^\infty(\Om)} \\
	&\le& - \frac{c_3 c_4}{|\Om|} \io v_k
	\qquad \mbox{for all $t>0$ and } k>k_\star.
  \eas
  Consequently, again by (\ref{41.9}), and by (\ref{41.10}), writing $c_6:=\frac{c_3 c_4}{|\Om|}$ we have
  \be{41.100}
	\|v_k\|_{L^\infty(\Om)}
	\le \frac{1}{c_3} \io v_k
	\le \frac{1}{c_3} \cdot \bigg\{ \io v_{0k} \bigg\} \cdot e^{-c_6 t}
	\le \frac{c_5}{c_3} e^{-c_6 t}
	\qquad \mbox{for all $t>0$ and } k>k_\star,
  \ee
  so that since (\ref{41.72}) warrants that $\|v_{k_l}(\cdot,t)\|_{L^\infty(\Om)} \to \|v(\cdot,t)\|_{L^\infty(\Om)}$
  for all $t>0$ as $l\to\infty$, by integrability of $(0,\infty) \ni t\mapsto \frac{c_5}{c_3} e^{-c_6 t}$ it follows
  from the dominated convergence theorem that $(0,\infty)\ni t\mapsto \|v(\cdot,t)\|_{L^\infty(\Om)}$ lies in $L^1((0,\infty))$
  with
  \bas
	J_l:=\int_0^\infty \|v_{k_l}(\cdot,t)\|_{L^\infty(\Om)} dt
	\to \int_0^\infty \|v(\cdot,t)\|_{L^\infty(\Om)} dt	=:J
	\qquad \mbox{as } l\to\infty,
  \eas
  and that for
  \bas
	\phi_l(t):=\frac{1}{J_l} \int_0^t \|v_{k_l}(\cdot,t')\|_{L^\infty(\Om)} dt',
	\qquad t\ge 0, \ l\in\N,
  \eas
  and
  \bas
	\phi(t):=\frac{1}{J} \int_0^t \|v(\cdot,t')\|_{L^\infty(\Om)} dt',
	\qquad t\ge 0,
  \eas
  we have
  \be{41.11}
	\phi_l\to \phi
	\quad \mbox{uniformly in } (0,\infty)
	\qquad \mbox{as } l\to\infty.
  \ee
  Since (\ref{41.4}) and (\ref{41.8}) guarantee that $v>0$ in $\bom\times [0,\infty)$, and that thus not only all the $\phi_l$ for 
  $l\in\N$ but also $\phi$ are strictly increasing elements of $C^1([0,\infty);[0,1))$ which for $t\ge 0$ satisfy
  \bas
	\phi_l'(t) = \frac{1}{J_l} \|v_{k_l}(\cdot,t)\|_{L^\infty(\Om)} > 0
	\quad \mbox{for all $l\in\N$}
	\qquad \mbox{and} \qquad
	\phi'(t)=\frac{1}{J} \|v(\cdot,t)\|_{L^\infty(\Om)} >0,
  \eas
  from (\ref{41.11}) and a well-known result from real analysis (\cite[Ch.~7]{rudin}) we obtain that
  \be{41.12}
	\phi_l^{-1}(s) \to \phi^{-1}(s)
	\quad \mbox{for all $s\in (0,1)$}
	\qquad \mbox{as } l\to\infty.
  \ee
  Writing
  \bas
	w_l(x,s):=u_{k_l}\big( x,\phi^{-1}(s)\big),
	\qquad x\in\bom, \ s\in [0,1), \ l\in\N,
  \eas
  by means of (\ref{41.8}) we therefore see that
  \be{41.13}
	w_l(x,s) \to w(x,s):=u\big(x,\phi^{-1}(s)\big)
	\qquad \mbox{for all $x\in \Om$ and $s\in (0,1)$}
	\qquad \mbox{as } l\to\infty.
  \ee
  Apart from that, based on (\ref{0}) we find that in the classical sense, for each $l\in\N$ we have
  \bas
	\lball
	w_{ls} = \frac{1}{2} \na\cdot \big(a_l(x,s)\na w_l^2\big)
	- \na \cdot \big( b_l(x,s) w_l^2\big)
	+ \ell_{k_l} a_l(x,s) w_l,
	\qquad & x\in\Om, \ s\in (0,1), \\[1mm]
	\big(\frac{1}{2} a_l(x,s) \na w_l^2 - b_l(x,s) w_l^2\big)\cdot\nu=0,
	\qquad & x\in\pO, \ s\in (0,1), \\[1mm]
	w_l(x,0)=u_{0 k_l}(x),
	\qquad & x\in\Om,
	\ear
  \eas
  where (\ref{41.9}), (\ref{41.43}) and the boundedness of $(\chi_k)_{k\in\N}$
  provide $c_7>0, c_8>0$ and $c_9>0$ such that if for $x\in\Om, s\in (0,1)$ and $l\in\N$ we let 
  \bas
	a_l(x,s):=J_l\cdot\frac{v_{k_l}(x,t)}{\|v_{k_l}(\cdot,t)\|_{L^\infty(\Om)}}
	\qquad \mbox{and} \qquad
	b_l(x,s):=\chi_{k_l} J_l \cdot \frac{v_{k_l}(x,t)\na v_{k_l}(x,t)}{\|v_{k_l}(\cdot,t)\|_{L^\infty(\Om)}}
	\qquad \mbox{with } t=\phi_l^{-1}(s),
  \eas
  then
  \bas
	c_7 \le a_l(x,s) \le c_8
	\quad \mbox{and} \quad
	|b_l(x,s)|\le c_9
	\qquad \mbox{for all $x\in\Om, \ s\in (0,1)$ and } l\in\N.
  \eas
  Thanks to the boundedness of $(\ell_k)_{k\in\N}$,
  we may therefore again invoke \cite[Theorem 1.3 and Remark 1.4]{porzio_vespri} to find a nondecreasing
  $\om_2: [0,\infty) \to [0,\infty)$ such that $\om_2(\xi)\to 0$ as $\xi\to 0$, and that
  \bea{41.14}
	\big| w_l(x_1,s_1)-w_l(x_2,s_2)\big|
	&\le& \om_2\big( |x_1-x_2| + |s_1-s_2|^\frac{1}{2}\big) \nn\\
	& & \hs{-4mm}
	\qquad \mbox{for all $(x_1,s_1)\in\Om\times (0,1), (x_2,s_2)\in\Om\times (0,1)$ and $l\in\N$.}
  \eea
  In consequence, the convergence in (\ref{41.13}) is actually uniform in $\Om\times (0,1)$, whence observing that
  \bas
	& & \hs{-20mm}
	\big| u_{k_l}(x,t)-u(x,t)\big| \\
	&=& \Big| \big( w_l\big(x,\phi(t)\big) - w\big(x,\phi(t)\big)\big) 
		+ \big( w_l\big(x,\phi_l(t)\big) - w_l\big(x,\phi(t)\big)\big) \Big| \\
	&\le& \|w_l-w\|_{L^\infty(\Om\times (0,1))}
	+ \om_2\big(|\phi_l(t)-\phi(t)|^\frac{1}{2}\big)
	\qquad \mbox{for all $(x,t)\in\Om\times (0,\infty)$ and $l\in\N$}  
  \eas
  by (\ref{41.14}), recalling (\ref{41.11}) we infer that
  \be{41.15}
	\|u_{k_l}-u\|_{L^\infty(\Om\times (0,\infty))} \to 0
	\qquad \mbox{as } l\to\infty.
  \ee
  Since from (\ref{41.100}) it follows that
  \bas
	\sup_{l,l'\in\N} \sup_{t>0} \|v_{k_l}(\cdot,t)-v_{k_{l'}}(\cdot,t)\|_{L^\infty(\Om)}
	\le \sup_{l,l'\in\N} \|v_{k_l} - v_{k_{l'}}\|_{L^\infty(\Om\times (0,t_0))}
	+ \frac{2c_5}{c_3} e^{-c_6 t_0}
	\qquad \mbox{for all } t_0>0,
  \eas
  by means of (\ref{41.72}) we furthermore find that
  \bas
	\|v_{k_l} - v\|_{L^\infty(\Om\times (0,\infty))} \to 0
	\qquad \mbox{as } l\to\infty,
  \eas
  which together with (\ref{41.15}) confirms (\ref{41.8}).\abs
  To finally verify that $(u,v)$ is a locally bounded weak solution of (\ref{0}) in $\Om\times (0,\infty)$ with initial
  data $(u_0,v_0)$, we note that according to (\ref{0}), for each $k\in\N$ and any $\vp\in C_0^\infty(\bom\times [0,\infty))$
  fulfilling $\frac{\pa\vp}{\pa\nu}=0$ on $\pO\times (0,\infty)$ we have
  \bas
	- \int_0^\infty \io u_k \vp_t 
	- \io u_{0k} \vp(\cdot,0)
	&=& \frac{1}{2} \int_0^\infty \io u_k^2 \na v_k \cdot\na\vp
	+ \frac{1}{2} \int_0^\infty \io u_k^2 v_k \Del\vp \nn\\
	& & + \chi_k \int_0^\infty \io u_k^2 v_k \na v_k\cdot\na\vp
	+ \ell_k \int_0^\infty \io u_k v_k\vp,
  \eas
  so that using (\ref{41.01}) now in its full strength, and utilizing
  (\ref{41.1}), (\ref{41.71}), (\ref{41.72}) and (\ref{41.73}) we obtain on letting $k=k_l\to\infty$ here that indeed
  (\ref{bu}) holds.
  As (\ref{bv}) can be derived in quite a similar manner, it remains to show that $\na u^2 \in L^2(\Om\times (0,T);\R^n)$
  for any $T>0$.
  To this end, for $k\in\N$
  we multiply the first equation in (\ref{0}) by $u_k^2$, integrate by parts and use Young's inequality to see that
  \bas
	\frac{1}{3} \frac{d}{dt} \io u_k^3
	+ 2\io u_k^2 v_k |\na u_k|^2
	&=& 2\chi_k \io u_k^3 v_k \na u_k\cdot\na v_k
	+ \ell_k \io u_k^3 v_k \nn\\
	&\le& \io u_k^2 v_k |\na u_k|^2
	+ \chi_k^2 \io u_k^4 v_k |\na v_k|^2
	+ \ell_k \io u_k^3 v_k.
  \eas
  Once again thanks to the boundedness of $(\chi_k)_{k\in\N}$ and $(\ell_k)_{k\in\N}$ as well as
  (\ref{41.2}) and (\ref{41.43}), we can thus find $c_{10}>0$ such that
  \bas
	\frac{1}{3} \frac{d}{dt} \io u_k^3
	+ \io u_k^2 v_k |\na u_k|^2
	\le c_{10}
	\qquad \mbox{for all $t>0$ and } k\in\N,
  \eas
  so that, in particular,
  \bas
	\int_0^T \io u_k^2 v_k |\na u_k|^2
	\le \frac{1}{3} \io u_{0k}^3 + c_{10} T
	\qquad \mbox{for all $T>0$ and } k\in\N.
  \eas
  Due to (\ref{41.4}), this ensures that for each $T>0$ there exists $c_{11}(T)>0$ such that
  \bas
	\int_0^T \io u_k^2 |\na u_k|^2
	\le c_{11}(T)
	\qquad \mbox{for all } k\in\N,
  \eas
  whence (\ref{41.71}) and the lower semicontinuity of the norm in $L^2(\Om\times (0,T))$ with respect to weak convergence
  imply that indeed $\na u^2$ belongs to $L^2(\Om\times (0,T);\R^n)$.
  We may accordingly conclude as intended.
\qed
By means of one more approximation argument involving regularized variants of the initial data in Theorem \ref{theo33},
we can use Lemma \ref{lem41} to achieve the announced goal of this section in the case when $n\le 2$ and $\Om$ is convex:\abs
\proofc of Theorem \ref{theo33}. \quad
  We fix $K>0$ and let $\Lam(K):=\Lam_1(K,K)$, with $\Lam_1(\cdot,\cdot)$ taken from Lemma \ref{lem333}.
  In order to simultaneously address both the existence statement and the claim concerning continuous dependence,
  we suppose that we are given $\chi$, $\ell$ and $(u_0,v_0)$ fulfilling (\ref{chi_ell}), (\ref{i2}) and (\ref{33.92}) 
  as well as numbers $\chi_k\ge 0$ and $\ell_k\ge 0$ and functions $u_{0k} \in C^0(\bom)$
  and $v_{0k} \in W^{1,\infty}(\Om)$, $k\in\N$, such that 
  $u_{0k}\not\equiv 0$ as well as $u_{0k}\ge 0$ and $v_{0k}>0$ in $\bom$ for all $k\in\N$, and that 
  (\ref{33.777}) and (\ref{33.2}) hold,
  where in view of (\ref{chi_ell}) and (\ref{33.92}) we may assume upon passing to a subsequence if necessary that
  \bas
	\chi_k \in [0,K)
	\quad \mbox{and} \quad
	\ell_k\in [0,K)
	\qquad \mbox{for all $k\in\N$,}
  \eas
  and that
  \bas
	\|u_{0k}\|_{L^\infty(\Om)} < K,
	\quad
	\io u_{0k}>\frac{1}{K},
	\quad
	\|v_{0k}\|_{L^\infty(\Om)} < K
	\quad \mbox{and} \quad
	\Big\|\frac{\na v_{0k}}{v_{0k}}\Big\|_{L^\infty(\Om)} < K
	\qquad \mbox{for all $k\in\N$.}
  \eas
  By density of $W^{1,\infty}(\Om)$ in $C^0(\bom)$, for fixed $k\in\N$ we can then 
  choose $(u_{0kl})_{l\in\N} \subset W^{1,\infty}(\Om)$
  such that $u_{0kl} \not\equiv 0$ and $u_{0kl}\ge 0$ in $\Om$ for all $l\in\N$, that still
  \bas
	\|u_{0kl}\|_{L^\infty(\Om)} < K,
	\quad
	\io u_{0kl}>\frac{1}{K},
	\quad
	\|v_{0k}\|_{L^\infty(\Om)} < K
	\quad \mbox{and} \quad
	\Big\|\frac{\na v_{0k}}{v_{0k}}\Big\|_{L^\infty(\Om)} < K
	\qquad \mbox{for all $l\in\N$,}
  \eas
  and that
  \bas
	\|u_{0kl} - u_{0k}\|_{L^\infty(\Om)} \le \eps(k,l):=\frac{1}{k+l}
	\qquad \mbox{for all } l\in\N,
  \eas
  so that
  \be{33.4}
	\big\| (u_{0kl}+\eps(k,l)) - u_{0k} \big\|_{L^\infty(\Om)}
	= \eps(k,l) + \|u_{0kl}-u_{0k}\|_{L^\infty(\Om)}
	\le \frac{2}{k+l}
	\qquad \mbox{for all } l\in\N,
  \ee
  whence, in particular, $\big\| (u_{0kl}+\eps(k,l)) - u_{0k} \big\|_{L^\infty(\Om)} \to 0$ as $l\to\infty$.
  Since $\eps(k,l)\in (0,1)$ and $u_{0kl} \in W^{1,\infty}(\Om)$ for all $l\in\N$, and since $n\le 2$ and $\Om$ is convex,
  given any $k\in\N$ and $l\in\N$ we may therefore rely on Lemma \ref{lem333} to find a global classical solution $(u_{kl},v_{kl})$
  of (\ref{0}) with parameters $\chi_k$ and $\ell_k$ and initial data $(u_{0kl}+\eps(k,l),v_{0k})$ which satisfies
  \be{33.5}
	\|u_{kl}(\cdot,t)\|_{L^\infty(\Om)}
	\le \Lam_1(K,K)=\Lam(K)
	\qquad \mbox{for all } t>0.
  \ee
  Based on (\ref{33.4}) and (\ref{33.5}), again for fixed $k\in\N$ we may employ Lemma \ref{lem41} to see that
  there exists a locally bounded weak solution
  $(u_k,v_k)\in C^0(\bom\times [0,\infty))\times \big( C^0(\bom\times [0,\infty)) \cap C^{2,1}(\bom\times (0,\infty))\big)$ 
  of (\ref{0}) in $\Om\times (0,\infty)$ with parameters $\chi_k$ and $\ell_k$ and emanating from
  the initial data $(u_{0k},v_{0k})$, and that for this solution we have
  \bas
	\|u_{kl}-u_k\|_{L^\infty(\Om\times (0,\infty))}
	+ \|v_{kl}-v_k\|_{L^\infty(\Om\times (0,\infty))}
	\to 0
	\qquad \mbox{as } l\to\infty,
  \eas
  meaning that to each $\iota>0$ there corresponds some $l_0(\iota,k)\in\N$ satisfying
  \be{33.6}
	\big\| u_{kl_0(\iota,k)} - u_k\|_{L^\infty(\Om\times (0,\infty))} 
	\le \frac{\iota}{2}.
  \ee
  We next combine (\ref{33.4}) with our hypothesis in (\ref{33.2}) to find that
  \bas
	& & \hs{-20mm}
	\sup_{l\in\N} \Big\{ \big\| (u_{0kl}+\eps(k,l)) - u_0\big\|_{L^\infty(\Om)} + \|v_{0k}-v_0\|_{W^{1,\infty}(\Om)} \Big\} \\
	&\le& \sup_{l\in\N} \Big\{ \frac{2}{k+l} + \|u_{0k}-u_0\|_{L^\infty(\Om)} + \|v_{0k}-v_0\|_{W^{1,\infty}(\Om)} \Big\} \\
	&\le& \frac{2}{k} + \|u_{0k}-u_0\|_{L^\infty(\Om)} + \|v_{0k}-v_0\|_{W^{1,\infty}(\Om)} \\[2mm]
	&\to& 0
	\qquad \mbox{as } k\to\infty,
  \eas
  so that a second application of Lemma \ref{lem41}, again relying on (\ref{33.5}), confirms the existence of a 
  locally bounded weak solution 
  $(u,v) \in C^0(\bom\times [0,\infty))\times \big( C^0(\bom\times [0,\infty)) \cap C^{2,1}(\bom\times (0,\infty))\big)$ 
  of (\ref{0}) in $\Om\times (0,\infty)$, now with parameters $\chi$ and $\ell$ and initial data $(u_0,v_0)$, which is such that
  \be{33.69}
	\big\| u_{kl_0(\iota,k)} - u\|_{L^\infty(\Om\times (0,\infty))}
	+ \|v_{kl_0(\iota,k)} - v\|_{L^\infty(\Om\times (0,\infty))}
	\to 0
	\qquad \mbox{as } k\to\infty,
  \ee
  and which in view of (\ref{33.5}) thus particularly satisfies (\ref{33.93}).
  Now thanks to (\ref{33.69}), we can fix $k_0(\iota)\in\N$ in such a way that
  \bas
	\big\| u_{kl_0(\iota,k)} - u\big\|_{L^\infty(\Om\times (0,\infty))}
	+ \big\|v_{kl_0(\iota,k)} - v\big\|_{L^\infty(\Om\times (0,\infty))}
	\le \frac{\iota}{2}
	\qquad \mbox{for all } k>k_0(\iota),
  \eas
  which in conjunction with (\ref{33.6}) implies that
  \bas
	\|u_k-u\|_{L^\infty(\Om\times (0,\infty))}
	+ \|v_k-v\|_{L^\infty(\Om\times (0,\infty))}
	\le \frac{\iota}{2}+\frac{\iota}{2}=\iota
	\qquad \mbox{for all } k>k_0(\iota).
  \eas
  As $\iota>0$ was arbitrary, this establishes (\ref{33.3}) and thereby proves the theorem.
\qed
Straighforward adaptation of this argument finally covers any smooth but not necessarily convex subdomain of $\R^n$ when $n\le 3$, 
provided that either $\chi$ and $\ell$ or $v_0$ are adequately small:\abs
\proofc of Theorem \ref{theo34}. \quad
  We this time let $\del(K,\LL):=\del_1(K,\LL)$ and $\Lam(K,\LL):=\Lam_1(K,\LL)$, with $\del_1(\cdot,\cdot)$ and $\Lam_1(\cdot,\cdot)$
  as provided by Lemma \ref{lem333}.
  Noting that (\ref{34.1}) and (\ref{34.2}) ensure the validity of
  (\ref{333.2}) and (\ref{333.3}) for all sufficiently large $k\in\N$, we can derive the statement of this theorem from
  Lemma \ref{lem333} and Lemma \ref{lem41} in a way very similar to that seen in the proof of Theorem \ref{theo33}.
\qed

\bigskip

{\bf Acknowledgement.} \quad
  The author acknowledges support of the {\em Deutsche Forschungsgemeinschaft} (Project No.~444955436).
  He moreover declares that he has no relevant financial or non-financial interests to disclose.\abs
{\bf Data availability statement.} \quad
  Data sharing is not applicable to this article as no datasets were generated or analyzed during the current study.\abs
{\bf AI statement.} \quad
  No AI tools have been used during preparation of this manuscript.

{\small
}

\end{document}